\documentclass[11pt,letterpaper]{amsart}
\usepackage{amsfonts}
\usepackage{mathrsfs}
\usepackage[all]{xy}
\usepackage{amsthm}
\usepackage{latexsym}
\usepackage{amssymb}
\usepackage{graphicx}
\usepackage{verbatim}
\usepackage{fancyhdr}

\numberwithin{equation}{section}

\newtheorem{thm}[equation]{Theorem}
\newtheorem{lemma}[equation]{Lemma}
\newtheorem{prop}[equation]{Proposition}

\newtheorem{rmk}[equation]{Remark}

\newtheorem{ex}[equation]{Example}
\newtheorem{defi}{Definition}

\newcommand{\Sec}{\operatorname{Sec}}
\newcommand{\Z}[0]{\mathbb{Z}}

\newcommand{\R}[0]{\mathbb{R}}
\newcommand{\C}[0]{\mathbb{C}}
\newcommand{\N}[0]{\mathbb{N}}
\newcommand{\p}[0]{\mathbb{P}}

\newcommand{\LL}[0]{\mathscr{L}}

\begin{document}

\title{Secant Degree of Toric Surfaces and Delightful Planar Toric Degenerations}
\author{Elisa Postinghel}
\address{Centre of Mathematics for Applications, University of Oslo - P.O. Box 1053 Blindern, N0-0316 Oslo, Norway}
\email[1]{elisa.postinghel@cma.uio.no}
\email[2]{elisa.postinghel@gmail.com}

\thanks{The author was partially supported by Marie-Curie IT Network SAGA, [FP7/2007-2013] grant agreement PITN-GA-
2008-214584.}

\maketitle

\begin{abstract}
The $k$-secant degree is studied with a combinatorial approach. A planar toric degeneration of any projective toric surface $X$ corresponds to a regular unimodular  triangulation $D$ of the polytope defining $X$. If the secant ideal of the initial ideal with respect to $D$ coincides with the initial ideal of the secant ideal, then $D$ is said to be delightful and the $k$-secant degree of $X$ can be easily computed. All delightful triangulations of toric surfaces having sectional genus $g\leq1$ are completely classified and, for $g\geq2$, a lower bound for the $2$- and $3$-secant degree, by means of the combinatorial geometry and the singularities of non-delightful triangulations, is established.

 \end{abstract}
 

\section*{Introduction}

There is a long tradition within algebraic geometry that studies the dimension and the degree of $k$-secant varieties.
Let $X\subseteq \p^r$ be a projective, irreducible variety of dimension $n$. Its $k$-\emph{secant variety} $\textrm{Sec}_k(X)$ is defined to be the closure of the union of all the $\p^{k-1}$'s in $\p^r$ meeting $X$ in  $k$ independent points.  If $\textrm{Sec}_k(X)$ has the expected dimension $kn+k-1$, what is the number $\nu_k(X)$ of $k$-secant $\p^{k-1}$'s to $X$ intersecting a general subspace of codimension $kn+k-1$ in $\p^r$? This is a problem which is unsolved in general. 

Our approach to the problem of computing the number $\nu_k$ for toric varieties is the one of  Ciliberto,  Dumitrescu and  Miranda  \cite{CDM} that is close to that of  Sturmfels and  Sullivant \cite{SS}.
Given a projective toric surface $X$, we perform \emph{planar toric degenerations}, i.e., we consider \emph{regular unimodular triangulations} $D$ of the polytope $P$ which defines $X$. 
The ideal $\mathcal{I}_0$ of the central fiber is the monomial initial ideal of the ideal $\mathcal{I}_X$ of $X$ with respect to a suitable term order $\prec$ which  corresponds to the triangulation $D$ (see  \cite[Theor. 8.3]{St}): $\mathcal{I}_0=\textrm{in}_{\prec}(\mathcal{I}_X).$ 

In Section \ref{toric definitions} and Section \ref{secants def section}  we introduce the objects of our study: convex lattice polytopes, toric varieties,  toric degenerations and $k$-secant varieties, with particular attention to the problem of computing the $k$-secant degree of toric surfaces.

In Section \ref{delightfulness} we introduce the notion of $k$-delightful planar toric degenerations of toric varieties: if the $k$-secant ideal of the initial ideal $\mathcal{I}_0$  of $X$  with respect to the degeneration coincides with the initial ideal of the $k$-secant ideal of $X$, then the degeneration is $k$-delightful. Sturmfels and Sullivant proved  in \cite[Theor. 5.4]{SS} that if there exists a triangulation $D$ of the polytope $P$ defining $X$ with at least one \emph{skew} $k$-\emph{set}, i.e., a subset of $k$  triangles of $D$ that are pairwise disjoint, then the $k$-secant variety of $X$ has the expected dimension. Moreover the number of such skew $k$-sets  is a lower bound for the number $\nu_k(X)$, see Theorem \ref{old lower bound}. If equality holds, then $D$ is $k$-delightful and the flat limit of the $k$-secant variety is a union of linear subspaces of dimension $kn+k-1$, hence the $k$-secant degree is computed. This bound is almost never sharp, indeed $k$-delightful degenerations are rare.

In Section \ref{results secant} we approach the secant degree computation and we give a lower bound for $\nu_k$, for $k=2,3$. The main tool is keeping into account the singularities of the configuration $D$ and explaining how they produce $k$-delightfulness defect. 
Our results can be regarded as the beginning of a similar study for the $k$-secant varieties of toric surfaces for $k\geq 4$ and, in higher dimension, for $k\geq2$.

The problem of finding delightful triangulations of polytopes was raised by Sturmfels and Sullivant \cite[Sect. 5]{SS}. They explored the existence of such triangulations for Veronese varieties, Segre varieties and rational normal scrolls. In Section \ref{classification delightful g=0,1} we provide a classification of all delightful triangulations for toric surfaces with sectional genus $0$ and $1$.

\section{Convex lattice polytopes and toric varieties}\label{toric definitions}

\subsection{Census of polytopes in $\R^2$ with $g\leq1$}\label{class pol}

A lattice point in $\R^n$ is a point with integral coordinates. A lattice polytope in $\R^n$ is a polytope whose vertices are lattice points. 
The \emph{normalized Ehrhart polynomial} of a lattice polytope $P$ in $\R^n$ is the numerical function 
$E_P: \N  \rightarrow\N$, $ t \mapsto \#(tP\cap \Z^n)$.
It is known that $E_P$ is a polynomial of degree $\dim(P)$:
$
E_P=\sum_{i=0}^{\dim(P)} \frac{c_i}{i!}t^i.
$
The leading coefficient $c_{\dim(P)}$ is denoted by $\textrm{Vol}(P)$ and it is called the \emph{(normalized) volume} of $P$. If $\dim(P)=n$, we have
$\textrm{Vol}(P)=n! \cdot{V(P)},$
where $V(P)$ is the usual Euclidean volume of $P$ (see \cite[p. 36]{St}). If $\dim(P)=1$, $Vol(P)+1$ turns out to be equal to the number of lattice points enclosed by $P$.  If $\dim(P)=n=2$, we denote by $Area(P)$ the normalized volume of $P$.

Set $n=2$ and denote by $g$ the number of interior lattice points of a plane polytope.
In this section we recall the classification of all convex lattice polytopes in $\R^2$ with $g\leq 1$, due to Rabinowitz \cite{Ra}.  
To this end, we need to  define an equivalence relation between planar polytopes (see \cite[p. 18]{Du} or \cite[p. 1]{Ra}). 
An integral unimodular affine transformation, also known as an \emph{equiaffinity}, in the plane is a linear transformation followed by a translation such that, furthermore, the corresponding matrix has determinant $1$ and integral entries. For example the matrix
$$ \left(\begin{array}{cc} 1 & 1 \\ 0 & 1 \end{array}\right)$$
acts on a polytope by sending the point $(x,y)^T\in\R^2$ to the point $(x+y,y)^T\in\R^2$: the points on the $x$-axis are fixed, while the points on the axis $y=k$ are shifted by $k$ on the right as for example in the picture:
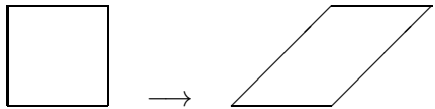
\begin{center}
\setlength{\unitlength}{0.333333mm}
\begin{picture}(40,40)(0,0)\put(0,0){\line(0,1){40}}\put(0,0){\line(1,0){40}}\put(0,40){\line(1,0){40}}\put(40,0){\line(0,1){40}}
\end{picture} 
 $\ \ \  \longrightarrow\ \ \ $
\begin{picture}(40,40)(0,0)
\put(0,0){\line(1,0){40}}\put(0,0){\line(1,1){40}}
\put(40,40){\line(1,0){40}}\put(40,0){\line(1,1){40}}
\end{picture}
\end{center}
Normalized area, number of lattice points and convexity of a plane polyotope are preserved under these transformations. Two plane polytopes are said to be \emph{lattice equivalent} if one can be transformed into the other via an equiaffinity, look for example to the above picture.

We will refer to \cite{Ra} for the proofs of the following results.

\begin{lemma}[The $x$-axis Lemma]
Let $q_1, q_2$ be the vertices of an edge of length $m$ of a polytope. There exists an equiaffinity
that maps $q_1$ into the origin, maps $q_2$ into the point $(m,0)$ on the positive x-axis, and maps
all the other vertices into points above the x-axis.
\end{lemma}

\begin{thm}[Characterization of polytopes with no interior lattice point]\label{census0}
If $P$ is a polytope with $g = 0$, then $P$ is lattice equivalent to one of the following:
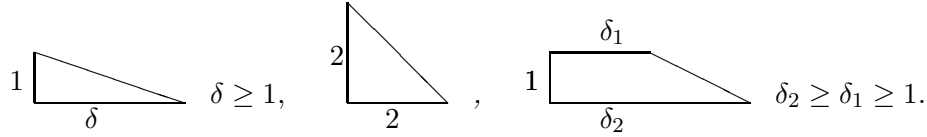
\begin{center}
\setlength{\unitlength}{0.333333mm}
 \begin{picture}(100,30)(0,-10)
\put(0,0){\line(0,1){20}}\put(0,0){\line(1,0){60}}\put(0,20){\line(3,-1){60}}\put(-10,5){\textsc{$1$}}\put(20,-10){\textsc{$\delta$}}\put(70,0){\textsc{$\delta\geq1$,}}    
\end{picture} \ \ \ \ \ 
\begin{picture}(60,50)(0,-10)
\put(0,0){\line(0,1){40}}\put(0,0){\line(1,0){40}}\put(0,40){\line(1,-1){40}}\put(-7,15){\textsc{$2$}}\put(15,-10){\textsc{$2$}}\put(50,0){,}
\end{picture}\ \ \ \ \ 
\begin{picture}(140,30)(0,-10)
\put(0,0){\line(0,1){20}}\put(0,0){\line(1,0){80}}\put(0,20){\line(1,0){40}}\put(40,20){\line(2,-1){40}}\put(-10,5){$1$}\put(-10,5){$1$}\put(20,25){$\delta_1$}\put(20,-10){$\delta_2$}\put(90,0){\textsc{$\delta_2\geq\delta_1\geq1$.}}
\end{picture}
\end{center}
\end{thm}

\begin{thm}[Characterization of polytopes with one interior lattice point]\label{census1}
If $P$ is a polytope with $g = 1$, then $P$ is lattice equivalent to one of the following:
\begin{itemize}
\item Triangles:
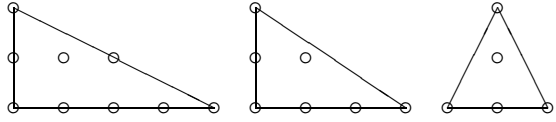
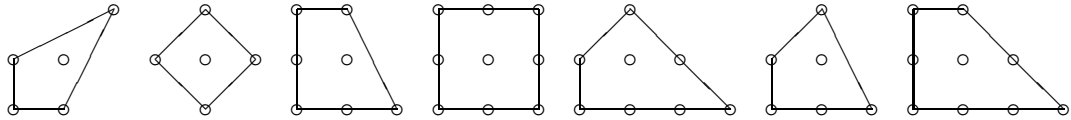
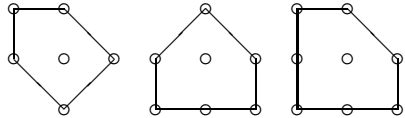
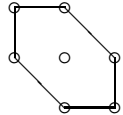
\begin{center}
\setlength{\unitlength}{0.333333mm}
\begin{picture}(80,45)(0,0)
\put(0,0){\line(0,1){40}}
\put(0,0){\line(1,0){80}}
\put(0,40){\line(2,-1){80}}
\put(-3,-3){$\circ$}\put(17,-3){$\circ$}\put(37,-3){$\circ$}\put(57,-3){$\circ$}\put(77,-3){$\circ$}
\put(-3,17){$\circ$}\put(17,17){$\circ$}\put(37,17){$\circ$}\put(-3,37){$\circ$}
\end{picture}\ \ \ \ 
\begin{picture}(60,45)(0,0)
\put(0,0){\line(0,1){40}}
\put(0,0){\line(1,0){60}}
\put(0,40){\line(3,-2){60}}
\put(-3,-3){$\circ$}\put(17,-3){$\circ$}\put(37,-3){$\circ$}\put(57,-3){$\circ$}
\put(-3,17){$\circ$}\put(17,17){$\circ$}\put(-3,37){$\circ$}
\end{picture}\ \ \ \ 
\begin{picture}(40,45)(0,0)
\put(0,0){\line(1,2){20}}
\put(0,0){\line(1,0){40}}
\put(40,0){\line(-1,2){20}}
\put(-3,-3){$\circ$}\put(17,-3){$\circ$}\put(37,-3){$\circ$}
\put(17,17){$\circ$}\put(17,37){$\circ$}
\end{picture}
\end{center}
\item Quadrilaterals:
\begin{center}
\setlength{\unitlength}{0.333333mm}
\begin{picture}(40,45)(0,0)
\put(0,0){\line(0,1){20}}
\put(0,0){\line(1,0){20}}
\put(0,20){\line(2,1){40}}
\put(20,0){\line(1,2){20}}
\put(-3,-3){$\circ$}\put(17,-3){$\circ$}
\put(-3,17){$\circ$}\put(17,17){$\circ$}\put(37,37){$\circ$}
\end{picture}\ \ \ \ 
\begin{picture}(40,45)(0,0)
\put(0,20){\line(1,1){20}}
\put(0,20){\line(1,-1){20}}
\put(40,20){\line(-1,1){20}}
\put(40,20){\line(-1,-1){20}}
\put(17,-3){$\circ$}\put(37,17){$\circ$}\put(17,17){$\circ$}\put(17,37){$\circ$}\put(-3,17){$\circ$}
\end{picture}\ \ \ \ 
\begin{picture}(40,45)(0,0)
\put(0,0){\line(0,1){40}}\put(0,40){\line(1,0){20}}\put(0,0){\line(1,0){40}}\put(20,40){\line(1,-2){20}}
\put(17,-3){$\circ$}\put(37,-3){$\circ$}\put(17,17){$\circ$}\put(17,37){$\circ$}\put(-3,17){$\circ$}
\put(-3,-3){$\circ$}\put(-3,37){$\circ$}
\end{picture}\ \ \ \ 
\begin{picture}(40,45)(0,0)
\put(0,0){\line(0,1){40}}
\put(0,0){\line(1,0){40}}
\put(40,0){\line(0,1){40}}
\put(0,40){\line(1,0){40}}
\put(17,-3){$\circ$}\put(37,-3){$\circ$}\put(17,17){$\circ$}\put(17,37){$\circ$}\put(-3,17){$\circ$}
\put(-3,-3){$\circ$}\put(-3,37){$\circ$}\put(37,17){$\circ$}\put(37,37){$\circ$}
\end{picture}\ \ \ \ 
\begin{picture}(60,45)(0,0)
\put(0,0){\line(0,1){20}}
\put(0,0){\line(1,0){60}}
\put(0,20){\line(1,1){20}}
\put(20,40){\line(1,-1){40}}
\put(17,-3){$\circ$}\put(37,-3){$\circ$}\put(17,17){$\circ$}\put(17,37){$\circ$}\put(-3,17){$\circ$}
\put(-3,-3){$\circ$}\put(57,-3){$\circ$}\put(37,17){$\circ$}
\end{picture}\ \ \ \ 
\begin{picture}(40,45)(0,0)
\put(0,20){\line(1,1){20}}\put(0,0){\line(0,1){20}}
\put(0,0){\line(1,0){40}}
\put(20,40){\line(1,-2){20}}
\put(17,-3){$\circ$}\put(37,-3){$\circ$}\put(17,17){$\circ$}\put(17,37){$\circ$}\put(-3,17){$\circ$}
\put(-3,-3){$\circ$}
\end{picture}\ \ \ \ 
\begin{picture}(60,45)(0,0)
\put(0,0){\line(0,1){40}}
\put(0,0){\line(1,0){60}}
\put(0,40){\line(1,0){20}}
\put(20,40){\line(1,-1){40}}
\put(17,-3){$\circ$}\put(37,-3){$\circ$}\put(17,17){$\circ$}\put(17,37){$\circ$}\put(-3,17){$\circ$}
\put(-3,-3){$\circ$}\put(57,-3){$\circ$}\put(37,17){$\circ$}\put(-3,37){$\circ$}
\end{picture}
\end{center}
\item Pentagons
\begin{center}
\setlength{\unitlength}{0.333333mm}
\begin{picture}(40,45)(0,0)
\put(20,0){\line(1,1){20}}
\put(20,0){\line(-1,1){20}}
\put(0,40){\line(1,0){20}}
\put(0,20){\line(0,1){20}}
\put(20,40){\line(1,-1){20}}
\put(17,-3){$\circ$}\put(37,17){$\circ$}\put(17,17){$\circ$}\put(17,37){$\circ$}\put(-3,17){$\circ$}\put(-3,37){$\circ$}
\end{picture}\ \ \ \ 
\begin{picture}(40,45)(0,0)
\put(0,0){\line(1,0){40}}
\put(0,0){\line(0,1){20}}
\put(40,0){\line(0,1){20}}
\put(0,20){\line(1,1){20}}
\put(20,40){\line(1,-1){20}}
\put(17,-3){$\circ$}\put(37,-3){$\circ$}\put(17,17){$\circ$}\put(17,37){$\circ$}\put(-3,17){$\circ$}
\put(-3,-3){$\circ$}\put(37,17){$\circ$}
\end{picture}\ \ \ \ 
\begin{picture}(40,45)(0,0)
\put(0,0){\line(1,0){40}}
\put(0,0){\line(0,1){40}}
\put(40,0){\line(0,1){20}}
\put(0,40){\line(1,0){20}}
\put(20,40){\line(1,-1){20}}
\put(17,-3){$\circ$}\put(37,-3){$\circ$}\put(17,17){$\circ$}\put(17,37){$\circ$}\put(-3,17){$\circ$}
\put(-3,-3){$\circ$}\put(37,17){$\circ$}\put(-3,37){$\circ$}
\end{picture}
\end{center}
\item Hexagons
\begin{center}
\setlength{\unitlength}{0.333333mm}
\begin{picture}(40,45)(0,0)
\put(0,20){\line(0,1){20}}
\put(0,20){\line(1,-1){20}}
\put(20,0){\line(1,0){20}}
\put(40,0){\line(0,1){20}}
\put(0,40){\line(1,0){20}}
\put(20,40){\line(1,-1){20}}
\put(17,-3){$\circ$}\put(37,-3){$\circ$}\put(17,17){$\circ$}\put(17,37){$\circ$}\put(-3,17){$\circ$}
\put(37,17){$\circ$}\put(-3,37){$\circ$}
\end{picture}
\end{center}
\end{itemize}
\end{thm}

The last quadrilateral was missing in the published paper \cite{Ra} and was later added to the classification.

We will use the notation $P^g(l,d,m)$ for these polytopes, where $g$ is the number of interior lattice points, $l$ is the number of edges (or vertices), $d$ is the normalized area and $m$ is the normalized maximal edge length. Actually we will denote in this way both the equiaffinity class and the representatives of the class, each time specifying what representative we are dealing with. The two quadrilaterals with $g=1$ and $l=d=4$ are not distinguished by this notation, because they both have $m=1$. So one could write $P^1(4,4,1)$ for the first one and $\tilde{P}^1(4,4,1)$ for the second one, but actually it does not matter since we will not deal with them in this paper.

\subsection{Toric varieties via polytopes and toric degenerations}

A convex lattice polytope $P$ in $\R^n$ defines a \emph{toric variety} $X_P$ of dimension $n$ endowed with an ample line bundle $\LL$ and therefore a morphism in $\p^r$, where $r+1$ equals the number of lattice points of $P$. Let $P \cap \Z^n=\{\underline{m}_0,\dots ,\underline{m}_r\}$ be the set of the lattice points of $P$, with $\underline{m}_i=(m_{i1},\dots,m_{in})$,  $i=0,\dots,r$. Consider the monomial map
$$
\begin{array}{llll}
\Phi_P: & (\C^{\ast})^{n} & \to &\p^r\\
\ & \underline{x}&\mapsto &[\underline{x}^{\underline{m}_0},\dots,\underline{x}^{\underline{m}_r}]
\end{array}
$$
where $\underline{x}=(x_1,\dots,x_n)$ and $\underline{x}^{\underline{m}_i}=x_1^{m_{i1}}\cdots x_n^{m_{in}}$. The projective toric variety $X_P\in \p^r$ is defined to be the closure  of the image of $\Phi_P$.
The degree of  $X_P$ equals the normalized volume
$\textrm{Vol}(P)$. Lattice equivalent polytopes in $\R^2$ define the same toric surface. 

 A \emph{subdivision} $D$ of $P$ is a partition of $P$ given by a finite family $\{Q_i\}_{i\in I}$ of convex sub-polytopes of maximal  dimension  such that 
\begin{itemize}
\item $\bigcup_{i\in I}Q_i=P$,
\item $Q_i\cap Q_j$, with $i\neq j$,  is either a common face or it is empty.
\end{itemize}

A subdivision $D$ is said to be \emph{regular} if there exists a piecewise linear positive function $F$ with values in $\R$ defined over $P$, verifying the following requests:
\begin{itemize}
\item each $Q_i$ is the orthogonal projection of the $n$-dimensional faces of the graph polytope $G(F):=\left\{(x,z)\in P\times \R: 0\leq z\leq F(x)\right\}$ of $F$ on $z=0$;
\item $F$ is \emph{strictly convex}.
\end{itemize}
We will call such an $F$ a \emph{lifting function} as in \cite{Hu}.
Given a regular subdivision $D$ of  $P$, we define the associated morphism as follows:
\begin{eqnarray}\label{morph D}
\begin{array}{llll}
\Phi_D: & (\C^{\ast})^{n}\times\C^\ast & \to &\p^r\times \C\\
\ & (\underline{x},t)&\mapsto &([t^{F(\underline{m}_0)}\underline{x}^{\underline{m}_0}:\cdots:t^{F(\underline{m}_r)}\underline{x}^{\underline{m}_r}],t)
\end{array}
\end{eqnarray}
The closure of $\Phi_D((\C^{\ast})^{n}\times\{t\})$, for all $t\neq 0$, is a variety $X_t$ projectively equivalent to  $X_P$. Let $X_0$ be the flat limit of $X_t$, when $t$ tends to zero: such a variety is the union of the varieties $X_{Q_i}$, $i\in I$. Indeed, the restriction $F_{|Q_i}$ of $F$ to $Q_i$ has equation $a_1 x_1+\cdots a_n x_n+b,$ for some $a_1,\dots,a_n,b\in \R$; we can always compose $\Phi_D$ with a reparametrization action of the torus $\C^{\ast}$,
$
x_1,\dots,x_n,t \mapsto t^{-a_1}x_1,\dots,t^{-a_n}x_n,t,
$
getting
$$
\begin{array}{lll}
 (\C^{\ast})^{n+1}& \to &\p^r\times\C\\
 (\underline{x},t)&\mapsto &([\cdots:t^{F(\underline{m}_i)-F_{Q_i}(\underline{m}_i)}\underline{x}^{\underline{m}_i}:\cdots],t).
\end{array}
$$
By letting $t\rightarrow0$, one sees that $X_{Q_i}$ sits in the flat limit $X_0$ of $X_t$.
The map (\ref{morph D}) can be extended to a map 
$$
\begin{array}{lll}
 X_P\times\C^{\ast} & \to &\p^r \times\C\\
 (\underline{x},t)&\mapsto &([t^{F(\underline{m}_0)}\underline{x}^{\underline{m}_0}:\cdots:t^{F(\underline{m}_r)}\underline{x}^{\underline{m}_r}],t)
\end{array}
$$
and the flat morphism
$
\pi_D:([t^{F(\underline{m}_0)}\underline{x}^{\underline{m}_0}:\cdots:t^{F(\underline{m}_r)}\underline{x}^{\underline{m}_r}],t) \mapsto t
$
provides a $1$-dimensional embedded degeneration of $X$ to  $X_0$.
$\pi_D$ is said to be a \emph{toric degeneration} of the toric variety $X_P$ and we will use the notation $X_0=\lim_D X$. 
The reducible central fiber $X_0$ is given by the subdivision $D$ of $P$: the irreducible components of $X_0$ are the $X_{Q_i}$'s. Notice that if $i\neq j$ and $Q_i$ and $Q_j$ have a common face $Q_i\cap Q_j$, then $X_{Q_i}$ and $X_{Q_j}$ intersect along $X_{Q_i\cap Q_j}$. 

If $n=2$ and the reducible central fiber $X_0$ is a union of planes, i.e. if the subdivision $D$ of the polytope $P$ is a regular unimodular triangulation of it, we say that $\pi_D$ is a \emph{planar toric degeneration} of $X_P$.
In this case the family $D$ of sub-polytopes of $P$ is a simplicial complex, whose maximal simplices are the $Q_i$'s.
The notion of toric degeneration to union of $\p^n$'s leads to the notion of term order. In fact there is a one-to-one correspondence between regular triangulations and term orders. Let $\prec$ be  any term order in
 $\C[x_0,\dots,x_r]$ and let $\mathcal{I}_0 := \textrm{in}_\prec(\mathcal{I})$ be the initial ideal of the ideal
 $\mathcal{I}$ of $X$. The radical of  $\mathcal{I}_0$ is a squarefree monomial ideal whose corresponding simplicial complex $\Delta_\prec(\mathcal{I}_0)$ is a regular triangulation of the polytope $P$ defining $X$. Conversely any regular triangulation of $P$ is of that form, for some $\prec$, see \cite[Theor. 8.3]{St}.

\section{Secant varieties}\label{secants def section}

Let $X \subset \p^{r}$ be an irreducible, non-degenerate, projective variety of dimension $n$. Fix an  integer $k\ge2$ and consider the $k$-th symmetric product $\textrm{Sym}^k(X)$. We define the \emph{abstract} $k$-\emph{th secant variety} of $X$, $S^k_X \subseteq \textrm{Sym}^k(X)\times \p^r$, as the Zariski closure of the set
$$
\{((x_1,\dots,x_k),z)\in \textrm{Sym}^k(X)\times \p^r : \textrm{dim}(\pi)=k-1 \textrm{ and } z \in\pi\}
$$
where $\pi=\langle x_1,\dots,x_k\rangle$. It is irreducible of dimension $kn+k-1$. 
Consider the projection $p^k_X$ on the second factor and define the $k$-\emph{th secant variety} of $X$, $\textrm{Sec}_k(X):=p^k_X(S^k_X)$,
as the image of $S^k_X$ in $\p^r$.
It is an irreducible algebraic variety of dimension $\textrm{dim}(\textrm{Sec}_k(X))\leq \textrm{min}\{kn+k-1,r\}$.
The right hand side is called the \emph{expected dimension} of $\textrm{Sec}_k(X)$. 
If strict inequality holds, $X$ is said to be $k$-\emph{defective}. 

The general fiber of $p^k_X$ is pure of dimension $kn+k-1-\textrm{dim(Sec}_k(X))$. Denote by $\mu_k(X)$ the number of irreducible components of this fiber. If $\textrm{dim(Sec}_k(X))=kn+k-1 \leq r,$ then $p^k_X$ is generically finite and $\mu_k(X)=\textrm{deg}(p^k_X)$, i.e., $\mu_k(X)$ is the number of $k$-secant $\p^{k-1}$'s to $X$ passing through the general point of $\textrm{Sec}_k(X)$ and it is called the $k$-\emph{secant order} of $X$, see \cite{CC2}. This number is equal to one unless $X$ is $k$-\emph{weakly defective}. The weakly defective surfaces are classified in  \cite{CC}.
Let $L$ be a general linear subspace of $\p^{r}$ of codimension $kn+k-1$: $X$ has  
$$\nu_{k}(X)=\mu_{k}(X)\cdot \textrm{deg}(\textrm{Sec}_{k}(X))$$
$k$-secant $\p^{k-1}$'s meeting $L$. Let $\pi_L$ be the projection of $X$ from $L$ to $\p^{kn+k-2}$:
the image of $X$ has $\nu_k(X)$ new $k$-secant $\p^{k-2}$'s that $X$ did not have.
The number $\nu_{k}(X)$ is called the \emph{number of apparent} $k$-\emph{secant} $\p^{k-2}$'\emph{s to} $X$.
In particular $\nu_2(X)$ corresponds to the number of double points that $X$ acquires in a general projection to $\p^{2n}$, $\nu_3(X)$ is the number of trisecant lines in a general projection of $X$ to $\p^{3n+1}$ and so on. 
Notice that if $\nu_{k}(X)=1$, then  $\textrm{Sec}_k(X)=\p^r$ and
  $\mu_{k}(X)=1$ which means that for a general points of $\textrm{Sec}_k(X)$ there is a unique $k$-secant $\p^{k-1}$.

Let $X$ be a smooth surface. Severi's \emph{double point formula} gives the  number of nodes of a general projection of  $X$ to $\p^4$:
$$
\nu_2(X)=\frac{d(d-5)}{2}-5g+6 p_a -K^2+11,
$$
where $d$ is the degree, $g$ is the sectional genus, $p_a$ is the arithmetic genus and $K$ is the canonical divisor of $X$.
In particular, if $X=X_P$ is a projective toric surface, then 
$$
\nu_2(X)=\frac{1}{2}(d^2-10d+5B+2V-12),
$$
where $d$ is the normalized area of the polytope $P$, $B$ is the number of lattice points on the boundary  and $V$ is the number of vertices of $P$, see \cite[Cor. 1.6]{CS}.

If $X$ is a surface not containing lines, a formula for $\nu_3(X)$, known as \emph{LeBarz' trisecant formula for  surfaces in} $\p^7$ (see \cite[p. 7]{LB} or \cite[p. 202]{Le}), is
$$
\nu_3(X)=\frac{1}{6}(d^3-30d^2+224d-3d(5HK+K^2-c_2)+192HK+56K^2-40c_2)
$$
where  $H$ is the hyperplane divisor and $c_2$ is the second Chern class of $X$. Moreover, if $X$ contains a finite number of lines, the contribution of each line to $\nu_2(X)$ is
$
-{{4+a}\choose{3}},
$
where $a \in \Z$ is its self-intersection. 
There are similar, but more complicated, formulas for the number $\nu_k(X)$ in the curve case (see \cite[Chapt. VIII]{ACGH}), and in the surface case, if $X$ does not contain any line, for $k \leq 5$ (see \cite{LB,Le}).

Unfortunately, the Severi's formula for $\nu_2(X)$ does not apply if $X$ is a singular surface. Moreover, in order to apply the formulas for $\nu_k(X)$, $k\geq3$, one needs to know how many lines are contained in $X$.  
In this paper we present a combinatorial framework for the study of the $k$-secant varieties to any projective toric surface, that makes the computation of $\nu_2$ and $\nu_3$ easier.


\subsection{The $k$-secant degree of toric surfaces with $g\leq1$}\label{secant degree}
In this section we will deal with the toric surfaces defined by the polytopes of Theorem \ref{census0} and Theorem \ref{census1}. They are all minimal $k$-secant degree surfaces, $\mathcal{M}^k$-surfaces (see \cite{CR}), i.e. $$\deg(\Sec_k(X))={{r-\dim(\Sec_k(X))+k}\choose{k}}.$$

\subsubsection{$g=0$}

The Veronese surface $V_2$ in $\p^5$ is described by the triangle $P^0(3,4,2)$. Its $2$-secant variety is a hypersurface of degree $3$. Moreover $\Sec_k(V_2) =\p^5$, $k\ge3$. 

Consider the  rational normal surface scroll $S=S(\delta_1,\delta_2)\subseteq\p^{\delta_1+\delta_2+1}$,  $\delta_1\leq\delta_2$, whose polytope is either the triangle $P^0(3,\delta_2,\delta_2)$ or the trapezium $P^0(4,\delta_1+\delta_2,\delta_2)$.
If $k\leq \delta_1$ and $3k-1\leq \delta_1+\delta_2+1$ then $S$ is non $k$-defective and has minimal $k$-secant degree, namely 
$
\deg(\Sec_k(S))={{\delta-2k+2}\choose{k}}
$
and $\mu_k(S)=1$, $k\geq 2$.
The ideal of these surfaces is generated by the $2\times2$-minors of a Hankel matrix. A determinantal presentation for the ideals of their $k$-secant varieties is known, see \cite[Prop. 2.2]{CJ}.

\subsubsection{$g=1$}

The $k$-secant varieties of the  three quartic toric surfaces in $\p^4$ defined by $P^1(3,4,2)$, $P^1(4,4,1)$ and $\tilde{P}^1(4,4,1)$  fill up $\p^4$, for each $k\geq2$. 

Let $V_3$ be the $3$-ple Veronese embedding of $\p^2$ in $\p^9$, described by the polytope $P^1(3,9,3)$. It is well known that it is non $k$-defective and is minimal $k$-secant degree for $k=2,3$.  
In particular $\Sec_2(V_3)$ has dimension $5$ and degree $15$, while $\Sec_3(V_3)$ has dimension $8$ and degree $4$. Moreover $\Sec_k(V_3)=\p^9$, $k\geq 4$.

The $i$-internal projections of $V_3$, i.e., the surfaces obtained from $V_3$ as projections from $i$ general points on it, $1\leq i\leq4$, are  del Pezzo surfaces of degree $9-i$ in $\p^{9-i}$. They are the ones defined by the subpolytopes of $P^1(3,9,3)$: $P^1(4,8,3)$, $P^1(4,7,3)$, $P^1(5,7,2)$, $P^1(3,6,3)$, $P^1(4,6,2)$, $P^1(5,6,2)$, $P^1(6,6,1)$, $P^1(4,6,2)$, $P^1(5,5,1)$.
For $k=2$, we have $\dim(\Sec_2(X))=5$ and $\nu_2(X)={{d-3}\choose2}$. For $k\geq3$, $\Sec_3(X)=\p^{9-i}$. In particular for the del Pezzo surface of degree $8$ in $\p^8$, that corresponds to $P^1(4,8,3)$, we have $\nu_3(X)=1$. 
All of them have ideals which are generated by quadrics and given by the $2\times2$ minors of a known matrix. Also the $k$-secant varieties, for $k=2,3$,  have a nice determinantal presentation: the equations are given by the $(k+1)\times(k+1)$ minors of the same matrix. For an overview see \cite{CGG,tesi}.

Let now $X,Y\subseteq\p^8$ be respectively the embedding of the smooth quadric $\p^1\times\p^1\subseteq\p^3$  and of the cone in $\p^3$ over a rational normal conic via the $2$-Veronese embedding. They correspond to $P^1(4,8,2)$ and $P^1(3,8,4)$ respectively. They both have $2$-secant variety of dimension $5$ and degree $10$. Moreover for both of them, the  $3$-secant variety has dimension $7$ and degree $4$, see \cite[Theor. 9.1]{CR}.


\section{$k$-delightful planar toric degenerations}\label{delightfulness}

Let $\mathcal{I}$ be an ideal in the polynomial ring $K[x_0,\dots,x_r]$. The \emph{secant} $\mathcal{I}^{\{2\}}=\mathcal{I}\ast\mathcal{I}$ of $\mathcal{I}$  is an ideal in  $K[x_0,\dots,x_r]$ defined in the following way: take the polynomial ring $K[\underline{x},\underline{y},\underline{z}]=K[x_0,\dots,x_r,y_0,\dots,y_r,z_0,\dots,z_r]$ and let $\mathcal{I}(\underline{y})$ and $\mathcal{I}(\underline{z})$ be the ideals obtained as images of $\mathcal{I}$ in $K[\underline{x},\underline{y},\underline{z}]$ via the maps $x_i\mapsto y_i$ and $x_i\mapsto z_i$,
for $i=0,\dots,r$. 
Then $\mathcal{I}^{\{2\}}$ is the elimination ideal $\left(\mathcal{I}(\underline{y})+\mathcal{I}(\underline{z})+\left\langle y_i+z_i-x_i:0\leq i\leq r\right\rangle\right)\cap K[x_0,\dots,x_r]$.
Similarly, we define the $k$-\emph{secant} of $\mathcal{I}$ as $\mathcal{I}^{\{k\}}=\mathcal{I}\ast\cdots\ast\mathcal{I}$.

For homogeneous prime ideals, the $k$-secant ideals represent the prime ideals of the $k$-secant varieties of irreducible projective varieties.

Let now $\prec$ be any term order. The initial ideal of the $k$-secant ideal $\mathcal{I}^{\{k\}}$ of $\mathcal{I}$ is contained in the  $k$-secant of the initial ideal  of $\mathcal{I}$, for $k\geq 1$:
\begin{eqnarray}\label{delightfulness SS}
\textrm{in}_{\prec}(\mathcal{I}^{\{k\}})\subseteq (\textrm{in}_{\prec}(\mathcal{I}))^{\{k\}}.
\end{eqnarray}
For a  reference see  \cite[Cor. 4.2]{SS}.
If equality holds in (\ref{delightfulness SS}), then $\prec$ is said to be $k$-\emph{delightful} for the ideal $\mathcal{I}$. It is said to be \emph{delightful} for $\mathcal{I}$ if it is $k$-delightful for $\mathcal{I}$, for every $k\geq1$.

For toric varieties this leads to the notion of delightful triangulations of polytopes.
Let $\pi_D$ be a toric degeneration of a toric variety $X$ of dimension $n$ to a union of $\p^n$'s.
Any subset of $D$ of $k$ pairwise skew $\p^n$'s, i.e. $k(n+1)$ vertices of $D$ such that they form the vertices of $k$ disjoint tetrahedra of $D$, $k\geq1$, will span a linear subspace of $\p^r$ of dimension $kn+k-1$. A subset  of this type is said to be a \emph{skew} $k$-\emph{set};  we denote by $N_k(D)$ the set of such skew $k$-sets and by  $\bar{\nu}_k(D)$ its cardinality, see \cite{CDM,SS}. Consider the following result, due to Sturmfels and Sullivant, which gives a lower bound to the number $\nu_k(X)$ for toric varieties.

\begin{thm}\cite[Theor. 5.4]{SS}
\label{old lower bound}
If there exists a toric  degeneration $\pi_D$ of $X$ to a union of $\p^n$'s for which there exists at least one skew $k$-set, then $\textrm{Sec}_k(X)$ has the expected dimension and $\nu_k(X)$ is bounded below by the number of skew $k$-sets:
\begin{equation}\label{formula delightful}
\nu_k(X)\geq \bar{\nu}_{k}(D).
\end{equation}
\end{thm}  
\begin{proof}
Notice first of all that  $kn+k-1\leq r$. Let $\mathcal{I}$ be the ideal of $X$ and let $\mathcal{I}_0$ be the ideal of the central fiber $X_0$ with respect to the toric degeneration $\pi_D$. The simplicial complex of $X_0$ is  $D$; let $D^{\{k\}}$ be the simplicial complex of $\mathcal{I}_0^{\{k\}}$: the simplices in $D^{\{k\}}$ are the unions of $k$ simplices in $D$, see \cite[Remark 2.9]{SS}. Notice that the simplices of $D^{\{k\}}$ of maximal dimension  are the skew $k$-sets and the subspaces they span sit in the flat limit of $\textrm{Sec}_k(X)$.
Therefore, if there exists at least one skew $k$-set in  $D$, then $\textrm{Sec}_k(X)$ has the expected dimension $kn+k-1$.

Notice that different skew $k$-sets could span the same subspace $\pi$ of $\p^r$ and that for the general point of $\pi$ there is a unique subspace of dimension $k-1$ meeting the $k$ planes each in a point, for each skew $k$-set spanning $\pi$. The toric variety described by $D^{\{k\}}$ is the reduced union of the coordinate subspaces in $\p^r$ given by the skew $k$-sets.
 Furthermore, the limit of the $k$-secant variety of $X$ contains the variety defined by the $k$-secant of $\mathcal{I}_0$ by (\ref{delightfulness SS}). This concludes the proof.
\end{proof}

Sturmfels and Sullivant in \cite{SS} conjectured that if equality holds in the lower bound in (\ref{formula delightful}), then the term order corresponding to the triangulation $D$ is  $k$-delightful. We will call such degenerations $k$-delightful, according to \cite{CDM}.

\begin{defi}
Let $P$ and $D$ be as above. If $\dim(\Sec_k(X_P))=kn+k-1\leq r$ and equality holds in (\ref{formula delightful}), then $D$ is said to be $k$-delightful. Moreover $D$ is said to be delightful if it is $k$-delightful for every $k$.
\end{defi}

Now, consider the examples in Figure \ref{esempi non-del}.
\begin{figure}[!h]
\setlength{\unitlength}{0.422222mm}$D$
\begin{picture}(30,30)(0,0)
\put(15,0){\line(1,0){15}}
\put(0,30){\line(1,0){15}}
\put(0,15){\line(0,1){15}}
\put(30,0){\line(0,1){15}}
\put(0,15){\line(1,-1){15}}
\put(15,30){\line(1,-1){15}}
\put(15,0){\line(0,1){30}}
\put(0,15){\line(1,0){30}}
\put(30,0){\line(-1,1){30}}
\put(13,13){$\bullet$}
\end{picture}\ \ \ \ \ \ \ \ \ $D'$
\begin{picture}(45,40)(0,0)
\put(0,0){\line(1,0){45}}
\put(0,15){\line(1,0){30}}
\put(0,30){\line(1,0){15}}
\put(0,0){\line(0,1){45}}
\put(15,0){\line(0,1){30}}
\put(30,0){\line(0,1){15}}
\put(0,45){\line(1,-1){45}}
\put(0,30){\line(1,-1){30}}
\put(0,15){\line(1,-1){15}}
\put(13,13){$\bullet$}
\end{picture}\caption{Non-$2$-delightful triangulations}\label{esempi non-del}
\end{figure}
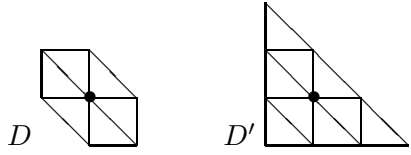
The first picture represents a triangulation $D$ of the hexagon $P^1(6,6,1)$, i.e., a degeneration of the smooth del Pezzo surface $X\subseteq\p^6$ to a union of six planes intersecting at a point. Since  $\bar{\nu}_2(D)=0$ and $\nu_2(X)=3$, $D$ is not $2$-delightful. The second one represents a triangulation  of the polytope $P^1(3,9,3)$ defining the Veronese surface $X'$ in $\p^9$. $\bar{\nu}_2(D')=12$ and $\nu_2(X')=15$ hence $D'$ is not $2$-delightful. 
Notice that in both cases there is a $2$-delightfulness defect equal to $3$. It is natural to wonder if the cause has to be sought in the sextuple central point, marked in the figures, that of course prevents the presence of disjoint triangles in the configurations. More generally, how do the singularities of the configuration influence the delightfulness property? This question was asked by Ciliberto, Dumitrescu and Miranda \cite{CDM}.
Our aim is to give an explanation of this phenomenon. In the next section we will propose our results in this direction.


\section{A lower bound for $\nu_k$, $k=2,3$}\label{results secant}

Let $P\subseteq\R^2$ be the defining polytope of a projective toric  surface $X$ and let $\pi_D$ be a (planar) toric degeneration of $X$ to a union of planes $X_0$. Let $p\in P\cap\Z^n$ be a lattice point of $P$ and let  $Q^1,\dots,Q^\delta \in D$ be the triangles in $D$ covering $p$: $Q^1\cap \cdots \cap Q^\delta=\{p\}$. Suppose that the union of the $Q^i$'s is a convex planar figure, namely a sub-polytope $Q_p$ of $P$. $Q_p$ has (normalized) area
$\delta.$
Let $Z=Z_p$ be the projective toric surface of degree $\delta$ defined by $Q_p$ and let $Z_{0}$ be the union of $\delta$ planes defined by the $Q^i$'s. 
If $p$ is a boundary lattice point, i.e. $Q_p$ has $g=0$,  we will call it a \emph{rational singularity} for $D$ because  $Z_{0}$  is a reduced chain of planes intersecting at a point (corresponding to $p$). If $p$ is an interior point, i.e. $Q_p$ has $g=1$, we will say that $p$ is an \emph{elliptic singularity} for $D$ since  the general hyperplane section of $Z_{0}$ is a cycle of lines. In Table \ref{caso razionale sec 1} and  Table \ref{caso ellittico sec 1} all these singularities are classified.

This section is devoted to the proof of the following result that improves the lower bound for $\nu_k$ of Proposition \ref{old lower bound} for the case $n=2$, $k=2,3$.

\begin{thm}\label{general sum}
Let $k\in\{2,3\}$. Let $X=X_P$ be a projective toric surface such that $\dim(\textrm{Sec}_k(X))=3k-1$.
Let $D$ be any triangulation of $P$. Let $\{p_i\}_{i \in I}\subseteq P\cap\Z^n$, $\{Q_{p_i}\}_{i \in I}$ and $\{Z_{p_i}\}_{i \in I}$ be as above. Assume that
\begin{enumerate}
\item $\dim\textrm{Sec}_k(Z_{p_i})=3k-1$, for $i\in I$,
\item there exists a regular subdivision $D^1_i$ of $P$ containing  $Q_{p_i}$. 
\end{enumerate}
Then $D$ is not $k$-delightful. Moreover
\begin{eqnarray}\label{senza overlap}
\nu_k(X)\geq \bar{\nu}_{k}(D)+ \sum_{i \in I} \nu_k(Z_{p_i}).
\end{eqnarray}
\end{thm}

\begin{rmk} This result can not be generalized  to the higher-order secant case. 
Let $k\ge4$. The expected dimension of $\textrm{Sec}_k(X)$ is $\min\{3k-1,r\}$, when $X\subseteq\p^r$ is a projective toric surface. 
None of the rational or elliptic sub-polytopes is interesting in this case, because $\dim(\textrm{Sec}_k(Z_p))<\dim(\textrm{Sec}_k(X))$, for any $Z_p$ as in Table \ref{caso razionale sec 1} or  Table \ref{caso ellittico sec 1}.
\end{rmk}

\subsection{Proof of Theorem \ref{general sum}}

\subsubsection{$k=2$}

Let $X=X_P$ be a projective toric surface such that $\dim\textrm{Sec}(X)=5$. Let $\pi_D$ be a planar toric degeneration of $X$  and let $p$ be a rational or elliptic singularity for $D$. Let $Q=Q_p
=P^0(l,\delta,m)$ be the sub-polytope of $P$ corresponding to $p$ and let $Z=Z_p$ be the projective toric surface of degree $\delta$ defined by $Q$: $Z\subseteq \p^{\delta'}\subseteq \p^r$, where
$$
\delta'=\left\{\begin{array}{ll}
\delta+1 & \textrm{if } p \textrm{ is rational}\\
\delta & \textrm{if } p \textrm{ is elliptic.}\\
\end{array}\right.
$$ 

We are going to  prove that the flat limit of the secant variety of $X$ has a $5$-dimensional component of degree $\nu_2(Z)$.  For this reason, we assume that $\delta'\geq 5$ so that $\dim(\Sec_2(Z_p))=5$ (cf. Section \ref{secant degree}).
Furthermore we assume that a lifting function $F_{D^1}$ over an intermediate partition $D^1$ of $P$, that contains $Q$ and other polytopes obtained as union of triangles of $D$, 
 exists. We propose a couple of examples in  Figure  \ref{decomposed deg lati=1} and in  Figure \ref{decomposed deg razio}. The  existence of such an $F_{D^1}$ will be discussed  in Subsection \ref{existence D^1}.
$D^1$ defines a degeneration $\pi_{D^1}$ of $X$ to a reducible surface that has $Z$ as component.
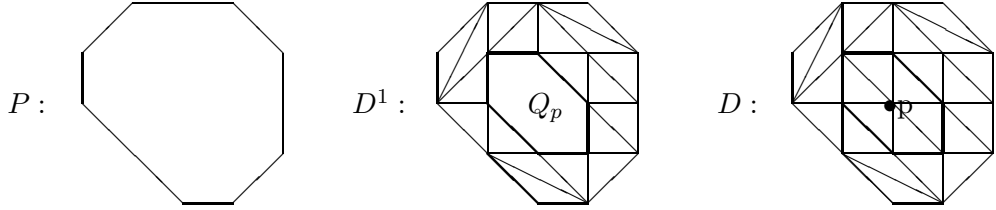
\begin{figure}[h!]
\begin{center}\ \ \ \ 
\setlength{\unitlength}{0.666666mm} 
\begin{picture}(40,40)(0,0)
\put(20,0){\line(1,0){10}}\put(0,20){\line(1,-1){20}}\put(0,20){\line(0,1){10}}\put(0,30){\line(1,1){10}}\put(10,40){\line(1,0){20}}\put(30,40){\line(1,-1){10}}\put(40,10){\line(0,1){20}}\put(30,0){\line(1,1){10}}\put(-15,18){$P:$}
\end{picture}  \ \ \ \ \ \ \ \ \ \ \ \ \ \ 
\setlength{\unitlength}{0.666666mm} 
\begin{picture}(40,40)(0,0)
\put(20,0){\line(1,0){10}}\put(0,20){\line(1,-1){20}}\put(0,20){\line(0,1){10}}\put(0,30){\line(1,1){10}}\put(10,40){\line(1,0){20}}\put(30,40){\line(1,-1){10}}\put(40,10){\line(0,1){20}}\put(30,0){\line(1,1){10}}\put(20,30){\line(1,-1){10}}\put(20,10){\line(1,0){10}}\put(10,30){\line(1,0){10}}\put(10,20){\line(0,1){10}}\put(30,10){\line(0,1){10}}\put(10,10){\line(1,0){10}}\put(30,10){\line(1,0){10}}\put(30,20){\line(1,0){10}}\put(20,30)
{\line(1,0){10}}\put(0,20){\line(1,1){10}}\put(0,20){\line(1,2){10}}\put(30,0){\line(0,1){10}}\put(30,20){\line(0,1){10}}\put(20,30){\line(0,1){10}}\put(10,10){\line(0,1){10}}\put(30,20){\line(1,-1){10}}\put(30,30){\line(1,-1){10}}\put(30,30){\line(1,0){10}}\put(30,30){\line(-1,1){10}}\put(20,40){\line(2,-1){20}}\put(10,30){\line(0,1){10}}\put(10,30){\line(1,1){10}}\put(0,20){\line(1,0){10}}\put(10,10){\line(2,-1){20}}\put(10,20){\line(1,-1){20}}
\thicklines\put(20,10){\line(1,0){10}}\put(10,20){\line(1,-1){10}}\put(10,30){\line(0,-1){10}}\put(10,30){\line(1,0){10}}\put(20,30){\line(1,-1){10}}\put(30,10){\line(0,1){10}}\put(-17,18){$D^1:$}\put(18,18){$Q_p$}
\put(10,10){\line(1,-1){10}}
\end{picture} \ \ \ \ \ \ \ \ \ \ \ \ \ \ 
\setlength{\unitlength}{0.666666mm}
\begin{picture}(40,40)(0,0)
\put(20,0){\line(1,0){10}}\put(0,20){\line(1,-1){20}}\put(0,20){\line(0,1){10}}\put(0,30){\line(1,1){10}}\put(10,40){\line(1,0){20}}\put(30,40){\line(1,-1){10}}\put(40,10){\line(0,1){20}}\put(30,0){\line(1,1){10}}\put(20,30){\line(1,-1){10}}\put(20,10){\line(1,0){10}}\put(10,30){\line(1,0){10}}\put(10,20){\line(0,1){10}}\put(30,10){\line(0,1){10}}\put(10,10){\line(1,0){10}}\put(30,10){\line(1,0){10}}\put(30,20){\line(1,0){10}}\put(20,30)
{\line(1,0){10}}\put(0,20){\line(1,1){10}}\put(0,20){\line(1,2){10}}\put(30,0){\line(0,1){10}}\put(30,20){\line(0,1){10}}\put(20,30){\line(0,1){10}}\put(10,10){\line(0,1){10}}\put(30,20){\line(1,-1){10}}\put(30,30){\line(1,-1){10}}\put(30,30){\line(1,0){10}}\put(30,30){\line(-1,1){10}}\put(20,40){\line(2,-1){20}}\put(10,30){\line(0,1){10}}\put(10,30){\line(1,1){10}}\put(0,20){\line(1,0){10}}\put(10,10){\line(2,-1){20}}\put(10,20){\line(1,-1){20}}\put(10,10){\line(1,-1){10}}\put(18,18){$\bullet$p}
\put(10,20){\line(1,0){20}}\put(20,10){\line(0,1){20}}\put(10,30){\line(1,-1){20}}
\thicklines\put(20,10){\line(1,0){10}}\put(10,20){\line(1,-1){10}}\put(10,30){\line(0,-1){10}}\put(10,30){\line(1,0){10}}\put(20,30){\line(1,-1){10}}\put(30,10){\line(0,1){10}}\put(-15,18){$D:$}
\end{picture}\end{center}
\caption{An example of decomposed degeneration, $Q_p=P^1(6,6,1)$.}\label{decomposed deg lati=1} 
\end{figure}

\begin{figure}[!h]\begin{center}\ \ \ \ 
\setlength{\unitlength}{0.666666mm}
\begin{picture}(40,40)(0,0)
\put(10,0){\line(1,0){20}}\put(0,10){\line(1,-1){10}}\put(0,10){\line(0,1){20}}\put(0,30){\line(1,1){10}}\put(10,40){\line(1,0){20}}\put(30,40){\line(1,-1){10}}\put(40,10){\line(0,1){20}}\put(30,0){\line(1,1){10}} \put(-15,18){$P:$}
\end{picture} \ \ \ \ \ \ \ \ \ \ \ \ \ \ 
\setlength{\unitlength}{0.666666mm}
\begin{picture}(40,40)(0,0)
\put(10,0){\line(1,0){20}}\put(0,10){\line(1,-1){10}}\put(0,10){\line(0,1){20}}\put(0,30){\line(1,1){10}}\put(10,40){\line(1,0){20}}\put(30,40){\line(1,-1){10}}\put(40,10){\line(0,1){20}}\put(30,0){\line(1,1){10}}
\put(0,10){\line(1,0){40}}\put(0,20){\line(1,0){10}}\put(0,30){\line(1,0){30}}
\put(10,0){\line(0,1){40}}\put(30,30){\line(0,1){10}}\put(30,30){\line(1,-1){10}}
\put(0,20){\line(1,-1){10}}\put(0,30){\line(1,-1){10}}\put(10,20){\line(1,1){20}}\put(10,20){\line(2,1){20}}
\put(10,30){\line(1,1){10}}\put(10,30){\line(2,1){20}}\put(30,40){\line(1,-2){10}}\put(10,10){\line(1,1){20}}
\put(20,3){$Q_p$}\put(27,17){$S_1$}\put(10,17){$S_{1,1}$}
\thicklines\put(10,10){\line(1,0){30}}\put(10,0){\line(1,0){20}}\put(10,0){\line(0,1){10}}\put(30,0){\line(1,1){10}}
\put(-17,18){$D^1:$}
\end{picture} \ \ \ \ \ \ \ \ \ \ \ \ \ \ 
\setlength{\unitlength}{0.666666mm}
\begin{picture}(40,40)(0,0)
\put(10,0){\line(1,0){20}}\put(0,10){\line(1,-1){10}}\put(0,10){\line(0,1){20}}\put(0,30){\line(1,1){10}}\put(10,40){\line(1,0){20}}\put(30,40){\line(1,-1){10}}\put(40,10){\line(0,1){20}}\put(30,0){\line(1,1){10}}
\put(0,10){\line(1,0){40}}\put(0,20){\line(1,0){20}}\put(30,20){\line(1,0){10}}\put(0,30){\line(1,0){30}}\put(10,0){\line(0,1){40}}\put(20,0){\line(0,1){20}}\put(30,10){\line(0,1){30}}\put(0,20){\line(1,-1){20}}\put(0,30){\line(1,-1){10}}\put(20,0){\line(1,1){10}}\put(20,0){\line(2,1){20}}\put(30,20){\line(1,-1){10}}\put(30,20){\line(-1,-1){10}}\put(30,30){\line(1,-1){10}}\put(30,30){\line(-1,-1){10}}
\put(10,20){\line(1,1){20}}\put(10,20){\line(2,1){20}}\put(10,10){\line(1,1){10}}\put(20,10){\line(1,2){10}}
\put(10,30){\line(1,1){10}}\put(10,30){\line(2,1){20}}
\put(30,40){\line(1,-2){10}}
\thicklines\put(10,10){\line(1,0){30}}\put(10,0){\line(1,0){20}}\put(10,0){\line(0,1){10}}\put(30,0){\line(1,1){10}}
\put(18,-2){$\bullet$p}
\put(-15,18){$D:$}
\end{picture}\end{center}
\caption{An example of decomposed degeneration, $Q_p=P^0(4,5,3)$.}\label{decomposed deg razio}
\end{figure}
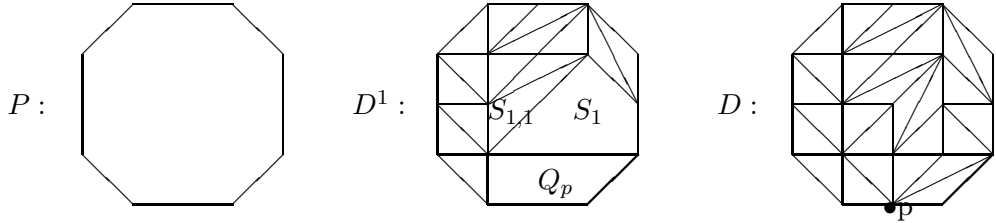

Let $\pi_{D^2}$ be the degeneration of the central fiber of $\pi_{D^1}$ to $X_0$.


\begin{prop}\label{new lower bound}
Keeping the same setting as above, if there exists in $D$ a singularity $p$ as in Table \ref{caso razionale sec 1}  or Table \ref{caso ellittico sec 1} and if there exists a regular subdivision $D^1$ of $P$ as above, then 
\begin{eqnarray}\label{our formula}
\nu_2(X)\geq \bar{\nu}_{2}(D)+ \nu_2(Z).
\end{eqnarray}
\end{prop}
\begin{proof}
Consider first the degeneration $D^1$ of $X$. Let $X^1_t$ be the fiber of $D^1$: $X^1_t\cong X$, for $t\neq 0$, while $X^1_0$ is the reduced union of the toric surfaces given by $D^1$. 
 We have that 
the secant variety of $Z$ and
 all the joins between components of  $X^1_0$
%
 sit in the flat limit  $\lim_{D^1}\textrm{Sec}(X)$ of the secant variety of $X$, with respect to $D^1$. 

 We consider now the second degeneration $D^2$ which has as general fiber $X^2_s\cong X^1_0$, $s\neq 0$, and as central fiber the reduced union of planes $X^2_0\cong X_0$. The flat limit, with respect to $D^2$, of $\lim_{D^1}\textrm{Sec}(X)$, that is $\lim_D\textrm{Sec}(X)$,   contains as component the flat limits, with respect to $D^2$, of all the components of $\lim_{D^1}\textrm{Sec}_2(X)$, namely  the following:
 $\lim_{D^2}\textrm{Sec}_2(Z)$, which is a $5$-dimensional component of degree $\nu_2(Z)$ and
the flat limit, with respect to $D^2$, of all the joins between components of $X^2_s$, $s\neq 0$. 
The union of these components contains the $\p^5$'s spanned by the elements of $N_2(D)$.

The contributions in terms of degree given by these components can be summed up.
 Indeed none of the $\p^5$'s spanned by the skew $2$-sets are contained in $\lim_{D^2}\textrm{Sec}_2(Z)$.
\end{proof}

If $\{p_i\}_{i \in I}$ are singularities of $D$ satisfying the hypotheses of Theorem \ref{new lower bound}, then the contributions given by $\nu_2(Z_{p_1})$'s do not interfere with each other. To see this, let us decompose the degeneration $D$ by taking subdivisions $D^1_i$ and $D^2_i$, for each $i$.  The flat limit of the secant variety of $Z_{p_i}$ with respect to $D^2_i$ sits in the flat limit of the secant variety of $X$ with respect to $D$, for every $i$, by Theorem \ref{new lower bound}. 
Furthermore, let $\p_i\subseteq\p^r$ be the projective subspace where $Z_{p_i}$, $\Sec_k(Z_{p_i})$ and their limits live, namely the space whose coordinate are given by the lattice points of $Q_{p_i}$. Notice that $\dim(\p_i\cap\p_j)\leq 3$, for all $i\neq j$. Indeed there are at most two coplanar triangles with vertices at two distinct points $p_i,p_j$. Since $(\lim_{D^2_i}\Sec_2(Z_{p_i}))\cap(\lim_{D^2_i}\Sec_2(Z_{p_i}))\subseteq \p_i\cap\p_j $, they have no common $5$-dimensional component. Therefore these limits are distinct components of $\lim_D\textrm{Sec}_2(X)$, for all $i,j \in I$, $i\neq j$. Furthermore all of them do not contain any element of $N_2(D)$,  hence the respective degrees sum up to $\bar{\nu}_2(D)$. This proves  Theorem \ref{general sum} for the case $k=2$.

\begin{ex}
Let $X$ be the quadric $\p^1\times\p^1$ embedded in $\p^{11}$ via $\mathcal{O}(2,3)$: $\nu_1(X)=\deg(\textrm{Sec}(X))=35$. Consider the two planar degenerations of $X$ shown in Figure \ref{ex2 sum}.
\begin{figure}[!h]
\begin{center}
\setlength{\unitlength}{0.666666mm} $D:$\ \  
\begin{picture}(30,20)(0,0)\put(0,0){\line(1,0){30}}\put(0,0){\line(0,1){20}}
\put(0,20){\line(1,0){30}}\put(0,10){\line(1,0){30}}
\put(10,0){\line(0,1){20}}\put(20,10){\line(0,1){10}}\put(0,10){\line(1,-1){10}}\put(0,20){\line(1,-1){10}}
\put(10,10){\line(1,1){10}}\put(10,0){\line(1,1){20}}
\put(10,0){\line(2,1){20}}\put(20,0){\line(1,1){10}}\put(30,0){\line(0,1){20}}
\put(8,8){$\bullet p_1$}\put(18,8){$\bullet p_2$}\put(8,-2){$\bullet p_3$}
\end{picture}\ \ \ \ \ \ $D':$ \ \ 
\begin{picture}(30,20)(0,0)\put(0,0){\line(1,0){30}}\put(0,0){\line(0,1){20}}
\put(0,20){\line(1,0){30}}\put(0,10){\line(1,0){30}}
\put(10,0){\line(0,1){20}}\put(20,10){\line(0,1){10}}\put(0,10){\line(1,-1){10}}\put(0,20){\line(1,-1){10}}
\put(10,20){\line(1,-1){10}}\put(10,0){\line(1,1){10}}\put(10,0){\line(2,1){20}}\put(20,0){\line(1,1){10}}
\put(20,20){\line(1,-1){10}}\put(30,0){\line(0,1){20}}
\put(8,8){$\bullet p'_2$}\put(18,8){$\bullet p'_3$}\put(8,-2){$\bullet p'_4$}\put(8,18){$\bullet p'_1$}\put(28,8){$\bullet q$}
\end{picture}
\caption{Triangulations of a rectangle}\label{ex2 sum}
\end{center}
\end{figure}
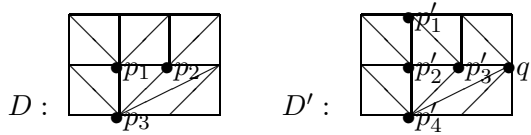
In the first case, the sum of the number of skew $2$-sets and of the contributions of the singularities restores the secant degree:
$
\bar{\nu}_2(D)+\nu_2(X_{p_1})+\nu_2(X_{p_2})+\nu_2(X_{p_3})=28+3+1+3=35.
$
In the second case we have:
$v_2(D')+\nu_1(X_{p'_1})+\nu_1(X_{p'_2})+\nu_1(X_{p'_3})+\nu_1(X_{p'_4}) =  29+1+1+1+1
 = 33<35$.
In $D'$ there is a lattice boundary point $q$ which is the common vertex of five triangles: certainly it causes an obstruction to the presence of skew $2$-sets, but the polygon given by the triangles around it is not convex and our argument does not apply.
\end{ex}

\subsubsection{$k=3$}

Let $X=X_P$ be a toric surface such that $\dim(\textrm{Sec}_2(X))=8$. Let $D$ be any triangulation of $P$.
\begin{rmk}
There are only two types of elliptic singularities we are interested in, namely the ones such that $Z_p$ is either the Veronese surface $V_3$ in $\p^9$ or the del Pezzo surface $X_8$ of degree eight in $\p^8$. Indeed in all remaining cases (see Table \ref{caso ellittico sec 1}) the $3$-secant variety has dimension less than $8$.
On the other hand, the only toric surface with $g=0$ such that  its $3$-secant variety has dimension $8$ and such that there exists a toric degeneration of it to a union of planes all of them  intersecting at a single point is the rational normal scroll $S(2,\delta-2)\subseteq\p^{\delta+1}$, with $\delta\geq7$, (see Table \ref{caso razionale sec 1}).
\end{rmk}

\begin{prop}\label{new lower bound 2}
Let $X=X_P$ be a toric surface such that $\dim\textrm{Sec}_3(X)=8$ and let $D$ be a triangulation of $P$. Let $p$ be a multiple point such that the corresponding surface $Z$ is either $V_3$, or $X_8$, or $S(2,\delta-2)$, with $\delta\geq 7$. Assume furthermore that  there exists an intermediate regular subdivision $D^1$ of $P$ containing  $Q_p$. Then
\begin{eqnarray}\label{our formula}
\nu_3(X)\geq \bar{\nu}_{3}(D)+ \nu_3(Z_p).
\end{eqnarray}
\end{prop}

\begin{proof}
 It is easy to see that 
 $\textrm{Sec}_3(Z)$ and
 $J(Y_i,J(Y_j,Y_l))$, where $Y_i,Y_j,Y_l$ are components of $\lim_{D^1}X$,
are in the flat limit $\lim_{D^1}\textrm{Sec}_2(X)$.

Then, looking at the second degeneration $D^2$, we see that the $\p^8$'s spanned by the skew $3$-sets of $D^2$ (that are the skew $3$-sets of $D$) and the limit $\lim_{D^2}\textrm{Sec}_3(Z)$ are $8$-dimensional of $\textrm{Sec}_3(X)$ with respect to $D$. 

Finally, the contributions $\bar{\nu}_3(D)$ and $\nu_3(Z)$ do not interfere with each other, following the same argument as in Theorem \ref{new lower bound}.
\end{proof}

If there are more than one singularity in $D$, $\{p_i\}_{i \in I}$, satisfying the  hypotheses of Theorem \ref{new lower bound 2}, 
arguing as for the case $k=2$, we get inequality (\ref{senza overlap}) for $k=3$.

\subsubsection{On the existence of an intermediate regular subdivision of a given triangulation}\label{existence D^1}
Let $P$, $D$ and $Q$ be as previously defined. 
To conclude this section we explore the existence of an intermediate regular subdivision $D^1$ containing $Q$. 

Assume first of all that either the edges of $Q$ have (normalized) length equal to one or they lie on the boundary of $P$ (under this assumption $p$ must be  an elliptic singularity). The family of sub-polytopes of $P$ given by $Q$ and by the $\textrm{Area}(P)-\delta$ remaining triangles of $D$ form a subdivision of $P$ (see Figure \ref{decomposed deg lati=1}). 
Such a subdivision is regular. Indeed, given a lifting function $F_D$ over $D$, one can always find a lifting function $F_{D^1}$ over $D^1$, exploiting the fact that strict convexity is a local property: it is enough to flatten $F_D$ over $Q$. More precisely, one can always assume that  
$
F_D(\underline{m})\gg 2, \textrm{ for }  \underline{m}\notin Q
$
and that
$$
F_D(\underline{m})=\left\{\begin{array}{ll}
 1-\epsilon & \textrm{ if } \underline{m}=p\\
 1 & \textrm{ if } \underline{m}\in Q\cap\Z^2\setminus \{p\}\\
 \end{array}
\right.,
$$
with $0<\epsilon \ll 1$.
Hence, a lifting function for $D^1$,   $F_{D^1}$, is the following:
$$
F_{D^1}(\underline{m}):=\left\{\begin{array}{ll}
1 & \textrm{ if } \underline{m}=p\\
F_{D}(\underline{m})& \textrm{ if }\underline{m}\neq p\\
 \end{array}
\right.
$$

Suppose now that $Q$ has edges $L_1\dots,L_s$, $s\leq l$ of length $>1$.
Let us construct  a partition of $P$ containing $Q$, triangles and convex polytopes given as union of triangles of $D$, using the following algorithm.\\
\textsc{Input}: a regular unimodular triangulation $D$ of $P$.\\
\textsc{Output}: a regular subdivision $D^1$ of $P$  containing  $Q$.
\begin{itemize}
\item[-] Let $S_i$ be the minimal convex union of triangles of $D$ such that $S_i\cap Q=L_i$, for $i=1,\dots,s$. If all the $S_i$'s have \emph{external} edges (i.e., all the edges except $L_i$) either of length one or lying on $\partial P$, we stop. 
 \item[-] Otherwise, let $L_{i,1},\dots,L_{i,s_i}$ be the external edges of $S_i$ of length $>1$, for $i\in\{1,\dots,s\}$. Let $S_{i,j}$ be the minimal convex union of triangles of $D$ such that $S_{i,j}\cap S_i=L_{i,j}$, $i=1,\dots,s$, $j=1,\dots,s_i$. If all the $S_{i,j}$'s have external edges either of length one or contained in $\partial P$, then we stop.
 \item[-] Otherwise we go on as above, until all the  polytopes obtained in this way have external edges either of length one, or contained in $\partial P$.
\end{itemize}
 
This process is finite. The output is a complex $D^1$ whose maximal polyhedra are $Q$, the $S_{i}$'s, the $S_{i,j}$'s, etc., and the remaining triangles of $D$. If one is able to flatten the lifting function $F_D$ over $Q$, the $S_{i}$'s, the $S_{i,j}$'s, etc., by rescaling it in such a way that the resulting piecewise linear function is strictly convex over $P$, one has found a lifting function $F_{D^1}$ for $D^1$ to be regular. 

At this point it is not difficult to define $D^2$: it is sufficient to take unimodular triangulations $D_Q$ of $Q$, $D_{S_i}$ of $S_i$, $D_{S_{i,j}}$ of $S_{i,j}$, etc., such that, combining them, one obtains the full regular unimodular triangulation $D$ of $P$. See for example Figure \ref{decomposed deg razio} to get an idea.


\section{Classification of delightful triangulations of polytopes with $g\leq1$}\label{classification delightful g=0,1}

In this section we classify all delightful triangulation of $g\leq1$ polytopes in $\R^2$.  
A necessary  condition for the degeneration to be $2$-delightful is that it contains no lattice point as in Table \ref{caso razionale sec 1} or  Table \ref{caso ellittico sec 1} in its configuration. Surprisingly we will see that the triangulations verifying this property turn out to be $k$-delightful, for any $k$.

\subsection{The rational case}
The $g=0$ polytopes are classified in Theorem \ref{census0}.  In this section we are going to prove the following theorem.

\begin{thm}\label{delightful class g=0}
The trapezium $P^0(4,2\delta+i,\delta+i)$ admits delightful triangulations if and only if $0\leq i\leq3$.

 The unique delightful triangulations of $P^0(4,2\delta+i,\delta+i)$, up to lattice equivalence, are the ones represented  in Figure \ref{delightful rational}.
\begin{figure}[h!!]
 \begin{center}
\begin{tabular}{ll}
\setlength{\unitlength}{0.333333mm}$i=0:$ &
$D_{\delta,\delta}\ \ $ 
\begin{picture}(140,25)(0,0)
\put(0,0){\line(0,1){20}}\put(20,0){\line(0,1){20}}\put(40,0){\line(0,1){20}}\put(60,0){\line(0,1){20}}\put(100,0){\line(0,1){20}}\put(120,0){\line(0,1){20}}
\put(0,0){\line(1,0){120}}
\put(0,20){\line(1,0){120}}
\put(0,20){\line(1,-1){20}}\put(20,20){\line(1,-1){20}}\put(40,20){\line(1,-1){20}}\put(100,20){\line(1,-1){20}}\put(75,5){$\cdots$}
\end{picture} \ \ \ \  
\setlength{\unitlength}{0.333333mm}$D'_{\delta,\delta}\ \ $
\begin{picture}(120,25)(0,0)
\put(0,0){\line(0,1){20}}\put(120,0){\line(0,1){20}}
\put(0,0){\line(1,0){120}}
\put(0,20){\line(1,0){120}}
\put(0,20){\line(1,-1){20}}\put(20,20){\line(1,-1){20}}\put(40,20){\line(1,-1){20}}\put(100,20){\line(1,-1){20}}
\put(0,20){\line(2,-1){40}}\put(20,20){\line(2,-1){40}}\put(80,20){\line(1,-1){20}}\put(80,20){\line(2,-1){40}}\put(65,5){$\cdots$}
\end{picture}\\
\setlength{\unitlength}{0.333333mm}$i=1:$ &
$D_{\delta,\delta+1}$
\begin{picture}(140,25)(0,0)
\put(0,0){\line(0,1){20}}\put(20,0){\line(0,1){20}}\put(40,0){\line(0,1){20}}\put(60,0){\line(0,1){20}}\put(100,0){\line(0,1){20}}\put(120,0){\line(0,1){20}}
\put(0,0){\line(1,0){140}}
\put(0,20){\line(1,0){120}}
\put(0,20){\line(1,-1){20}}\put(20,20){\line(1,-1){20}}\put(40,20){\line(1,-1){20}}\put(100,20){\line(1,-1){20}}\put(120,20){\line(1,-1){20}}\put(75,5){$\cdots$}
\end{picture}\ \ \ \ 
\setlength{\unitlength}{0.333333mm}$D'_{\delta,\delta+1}$
\begin{picture}(140,25)(0,0)
\put(0,0){\line(0,1){20}}\put(120,20){\line(1,-1){20}}\put(100,20){\line(2,-1){40}}
\put(0,0){\line(1,0){140}}
\put(0,20){\line(1,0){120}}
\put(0,20){\line(1,-1){20}}\put(20,20){\line(1,-1){20}}\put(40,20){\line(1,-1){20}}\put(100,20){\line(1,-1){20}}
\put(0,20){\line(2,-1){40}}\put(20,20){\line(2,-1){40}}\put(80,20){\line(1,-1){20}}\put(80,20){\line(2,-1){40}}\put(65,5){$\cdots$}
\end{picture}\\
\setlength{\unitlength}{0.333333mm}$i=2:$ & 
$D'_{\delta,\delta+2}$
\begin{picture}(160,25)(0,0)
\put(0,0){\line(0,1){20}}\put(120,20){\line(1,-1){20}}\put(100,20){\line(2,-1){40}}\put(120,20){\line(2,-1){40}}
\put(0,0){\line(1,0){160}}
\put(0,20){\line(1,0){120}}
\put(0,20){\line(1,-1){20}}\put(20,20){\line(1,-1){20}}\put(40,20){\line(1,-1){20}}\put(100,20){\line(1,-1){20}}
\put(0,20){\line(2,-1){40}}\put(20,20){\line(2,-1){40}}\put(80,20){\line(1,-1){20}}\put(80,20){\line(2,-1){40}}\put(65,5){$\cdots$}
\end{picture}
\\
\setlength{\unitlength}{0.333333mm}$i=3:$ & 
$D'_{\delta,\delta+3}$
\begin{picture}(180,25)(0,0)
\put(0,0){\line(0,1){20}}\put(120,20){\line(1,-1){20}}\put(100,20){\line(2,-1){40}}\put(120,20){\line(2,-1){40}}
\put(0,0){\line(1,0){180}}
\put(0,20){\line(1,0){120}}\put(120,20){\line(3,-1){60}}
\put(0,20){\line(1,-1){20}}\put(20,20){\line(1,-1){20}}\put(40,20){\line(1,-1){20}}\put(100,20){\line(1,-1){20}}
\put(0,20){\line(2,-1){40}}\put(20,20){\line(2,-1){40}}\put(80,20){\line(1,-1){20}}\put(80,20){\line(2,-1){40}}\put(65,5){$\cdots$}
\end{picture}
\end{tabular}
\end{center}
\caption{Delightful triangulations of $g=0$ polytopes}\label{delightful rational}
\end{figure}
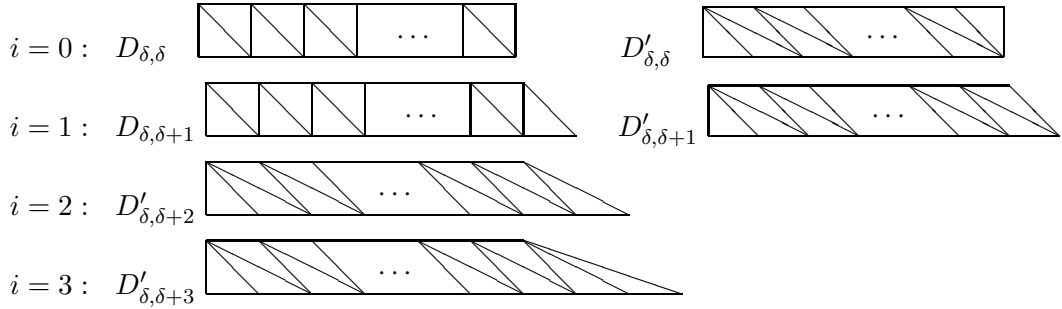\end{thm}

 The outline of the proof will be the following. 
As a first step we fix $k=2$ and we construct triangulations without rational singularities at the (boundary) lattice points for the polytopes $P^0(4,2\delta+i,\delta+i)$, $\delta\geq2$, $i\geq0$.  Then we will investigate their $k$-delightfulness.

\begin{rmk}
The unique triangulations of  $P^0(4,2\delta+i,\delta+1)$ without rational singularities occur when $0\leq i\leq 3$ and are the ones in Figure \ref{delightful rational}. 
\end{rmk}
\begin{proof}
 Consider the rectangle $P^0(4,2\delta,\delta)$ with bases of length $\delta$. 
 We start the triangulation in the only possible way (up to equiaffinity), as follows
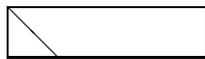
\begin{center}
\setlength{\unitlength}{0.333333mm}
\begin{picture}(80,20)(0,0)\put(0,0){\line(0,1){20}}\put(80,0){\line(0,1){20}}\put(0,20){\line(1,0){80}}
\put(0,0){\line(1,0){80}}\put(0,20){\line(1,-1){20}}\put(80,0){\line(0,1){20}}
\end{picture}\end{center}
Then there are only two distinct possibilities to add a further triangle that is adjacent to the previous one:
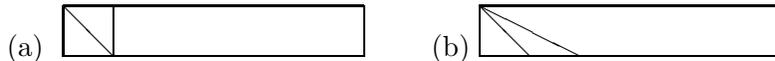
\begin{center}
(a)
\setlength{\unitlength}{0.333333mm}
\begin{picture}(120,20)(0,0)
\put(0,0){\line(0,1){20}}\put(120,0){\line(0,1){20}}\put(0,20){\line(1,0){120}}\put(0,0){\line(1,0){120}}\put(0,20){\line(1,-1){20}}\put(120,0){\line(0,1){20}}\put(20,0){\line(0,1){20}}
\end{picture}\ \ \ \ \ \ \  
(b)
\begin{picture}(120,20)(0,0)
\put(0,0){\line(0,1){20}}\put(120,0){\line(0,1){20}}\put(0,20){\line(1,0){120}}\put(0,0){\line(1,0){120}}\put(0,20){\line(1,-1){20}}\put(120,0){\line(0,1){20}}\put(0,20){\line(2,-1){40}}
\end{picture}
\end{center}
In case $(a)$, the ways of putting another triangle adjacent to the previous are the following:
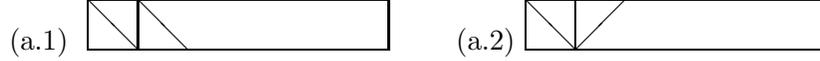
\begin{center}
(a.1)
\setlength{\unitlength}{0.333333mm}
\begin{picture}(120,20)(0,0)
\put(0,0){\line(0,1){20}}\put(120,0){\line(0,1){20}}\put(0,20){\line(1,0){120}}\put(0,0){\line(1,0){120}}\put(0,20){\line(1,-1){20}}\put(120,0){\line(0,1){20}}\put(20,0){\line(0,1){20}}\put(20,20){\line(1,-1){20}}
\end{picture}\ \ \ \ \ \ \  
(a.2)
\begin{picture}(120,20)(0,0)
\put(0,0){\line(0,1){20}}\put(120,0){\line(0,1){20}}\put(0,20){\line(1,0){120}}\put(0,0){\line(1,0){120}}\put(0,20){\line(1,-1){20}}\put(120,0){\line(0,1){20}}\put(20,0){\line(0,1){20}}\put(20,0){\line(1,1){20}}
\end{picture}
\end{center}
The second possibility must be excluded, otherwise we would get at least four triangles covering the point with coordinates $(1,0)$ and this certainly will generate a rational singularity (see Table \ref{caso razionale sec 1}). 
On the other hand, starting from the case (a.1) and  adding a triangle in the subdivision, we get 
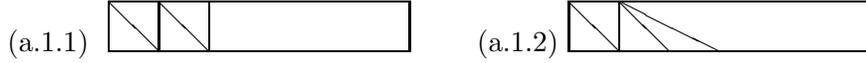
\begin{center}
(a.1.1)
\setlength{\unitlength}{0.333333mm}
\begin{picture}(120,20)(0,0)
\put(0,0){\line(0,1){20}}\put(120,0){\line(0,1){20}}\put(0,20){\line(1,0){120}}\put(0,0){\line(1,0){120}}\put(0,20){\line(1,-1){20}}\put(120,0){\line(0,1){20}}\put(20,0){\line(0,1){20}}\put(20,20){\line(1,-1){20}}\put(40,0){\line(0,1){20}}
\end{picture}\ \ \ \ \ \ \  
(a.1.2)
\begin{picture}(120,20)(0,0)
\put(0,0){\line(0,1){20}}\put(120,0){\line(0,1){20}}\put(0,20){\line(1,0){120}}\put(0,0){\line(1,0){120}}\put(0,20){\line(1,-1){20}}\put(120,0){\line(0,1){20}}\put(20,0){\line(0,1){20}}\put(20,20){\line(1,-1){20}}\put(20,20){\line(2,-1){40}}
\end{picture}\end{center}
The second configuration is excluded once again, otherwise the point $(1,1)$ would be covered by a chain of at least four triangles. So, iterating this argument, we obtain $D_{\delta,\delta}$.

In case (b), the possibilities are:
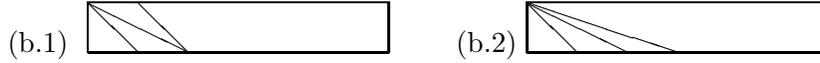
\begin{center}
(b.1)
\setlength{\unitlength}{0.333333mm}
\begin{picture}(120,20)(0,0)
\put(0,0){\line(0,1){20}}\put(120,0){\line(0,1){20}}\put(0,20){\line(1,0){120}}\put(0,0){\line(1,0){120}}\put(0,20){\line(1,-1){20}}\put(120,0){\line(0,1){20}}\put(0,20){\line(2,-1){40}}\put(20,20){\line(1,-1){20}}
\end{picture}\ \ \ \ \ \ \  
(b.2)
\begin{picture}(120,20)(0,0)
\put(0,0){\line(0,1){20}}\put(120,0){\line(0,1){20}}\put(0,20){\line(1,0){120}}\put(0,0){\line(1,0){120}}\put(0,20){\line(1,-1){20}}\put(120,0){\line(0,1){20}}\put(0,20){\line(2,-1){40}}\put(0,20){\line(3,-1){60}}
\end{picture}\end{center}
As above, the case (b.2) is excluded, otherwise $(0,1)$ would be a rational singularity. Then from (b.1) we obtain
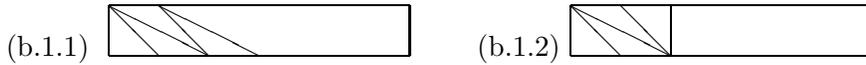
\begin{center}
(b.1.1)
\setlength{\unitlength}{0.333333mm}
\begin{picture}(120,20)(0,0)
\put(0,0){\line(0,1){20}}\put(120,0){\line(0,1){20}}\put(0,20){\line(1,0){120}}\put(0,0){\line(1,0){120}}\put(0,20){\line(1,-1){20}}\put(0,20){\line(2,-1){40}}\put(20,20){\line(1,-1){20}}\put(20,20){\line(2,-1){40}}
\end{picture}\ \ \ \ \ \ \  
(b.1.2) \begin{picture}(120,30)(0,0)
\put(0,0){\line(0,1){20}}\put(120,0){\line(0,1){20}}\put(0,20){\line(1,0){120}}\put(0,0){\line(1,0){120}}\put(0,20){\line(1,-1){20}}\put(120,0){\line(0,1){20}}\put(0,20){\line(2,-1){40}}\put(20,20){\line(1,-1){20}}\put(40,0){\line(0,1){20}}
\end{picture}
\end{center}
We exclude the case (b.1.2) and iterating the process we get $D'_{\delta,\delta}$  from (b.1.1).
The subdivisions $D_{\delta,\delta}$ and $D'_{\delta,\delta}$ do not contain any rational singularity and they are the unique triangulations of $P^0(4,2\delta,\delta)$ with this property.

Consider  $P^0(4,2\delta+1,\delta+1)$. One has two distinct ways (up to equiaffinity) to start a triangulation of this polytope:
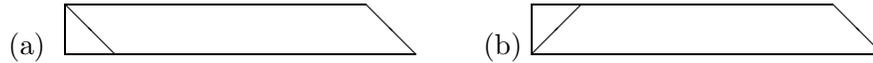
\begin{center}
(a)
\setlength{\unitlength}{0.333333mm}
\begin{picture}(140,20)(0,0)
\put(0,0){\line(0,1){20}}\put(140,0){\line(-1,1){20}}\put(0,20){\line(1,0){120}}\put(0,0){\line(1,0){140}}\put(0,20){\line(1,-1){20}}
\end{picture}\ \ \ \ \ \ \  
(b)
\begin{picture}(140,20)(0,0)
\put(0,0){\line(0,1){20}}\put(140,0){\line(-1,1){20}}\put(0,20){\line(1,0){120}}\put(0,0){\line(1,0){140}}\put(0,0){\line(1,1){20}}
\end{picture}
\end{center}
From (a), arranging the argument of above to this case, we arrive to $D_{\delta,\delta+1}$ or $D'_{\delta,\delta+1}$.  Instead, from (b) we get either
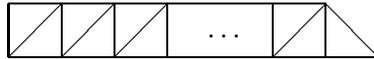
\begin{center}
\begin{picture}(140,20)(0,0)
\put(0,0){\line(0,1){20}}\put(20,0){\line(0,1){20}}\put(40,0){\line(0,1){20}}\put(60,0){\line(0,1){20}}\put(100,0){\line(0,1){20}}\put(120,0){\line(0,1){20}}
\put(0,0){\line(1,0){140}}
\put(0,20){\line(1,0){120}}
\put(0,0){\line(1,1){20}}\put(20,0){\line(1,1){20}}\put(40,0){\line(1,1){20}}\put(100,0){\line(1,1){20}}\put(120,20){\line(1,-1){20}}\put(75,5){$\cdots$}
\end{picture}\end{center}
that is lattice equivalent to $D'_{\delta,\delta+1}$, or
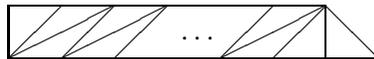
\begin{center}
\begin{picture}(140,20)(0,0)
\put(0,0){\line(0,1){20}}\put(0,0){\line(2,1){40}}\put(20,0){\line(2,1){40}}\put(80,0){\line(1,1){20}}\put(80,0){\line(2,1){40}}\put(120,0){\line(0,1){20}}
\put(0,0){\line(1,0){140}}
\put(0,20){\line(1,0){120}}
\put(0,0){\line(1,1){20}}\put(20,0){\line(1,1){20}}\put(40,0){\line(1,1){20}}\put(100,0){\line(1,1){20}}\put(120,20){\line(1,-1){20}}\put(65,5){$\cdots$}
\end{picture}\end{center}
that is excluded; in fact a singularity at the point $(\delta,1)$ has been generated.

Finally, arguing as above, we get $D'_{\delta,\delta+2}$
for the trapezium $P^0(4,2\delta+2,\delta+2)$ and $D'_{\delta,\delta+3}$
for $P^0(4,2\delta+3,\delta+3)$.
The details are easy and left to the reader.

If $i\geq 4$, it is not possible to find a triangulation without generating a rational singularity, because a chain of four triangles around a boundary lattice point will inevitably be created.
\end{proof}

Let now $P'\subseteq P$ be polytopes with $\textrm{Area}(P')+1=\textrm{Area}(P)=d$, $g(P)=g(P')=0$  and such that $P\setminus P'=T$ is a triangle of normalized area $1$. Let $D$ and $D'=D \setminus T$ be regular triangulations of $P$ and $P'$ respectively. Assume moreover that $\dim(\Sec_2(X_P))=\dim(\Sec_2(X_{P'}))=5$. Notice that, under these hypotheses, if $P$ belongs to the class $P^0(4,2\delta+i,\delta+i)$, $0\leq i\leq 3$, with $d=2\delta+i$ for some $\delta,i$,  then $P'$ has also the form $P^0(4,2\delta'+i',\delta+i')$, $0\leq i'\leq 3$, with $d-1=2\delta'+i'$ for some $\delta',i'$. Notice moreover that if $D$ is lattice equivalent to one of the configurations in Figure \ref{delightful rational}, then $D'$ is.


Define $D''=\{T''\in D': T''\cap T=\emptyset\}\subseteq D'\subseteq D$. $D''$ is given by those triangles of $D$ which do not intersect $T$. Using these notations we can describe $N_{k}(D)$ as the set given by the skew $k$-sets contained in $D'$ and by those involving $T$, namely
  $N_k(D)=N_k(D')\cup \{(T,(T''_1,\dots,T''_{k-1})):(T''_1,\dots,T''_{k-1})\in N_{k-1}(D'')\}$. 

\begin{lemma}\label{D' del implica D del}
In the above notation,  $D$ is $k$-delightful if and only if $D'$ is $k$-delightful and $D''$ is $(k-1)$-delightful.
\end{lemma}
\begin{proof}
Since $D''$ contains at most $d-3$ triangles, then $\bar{\nu}_{k-1}(D'')\leq {{(d-3)-2(k-2)}\choose {k-1}}$. Hence
$\bar{\nu}_{k}(D)=\bar{\nu}_{k}(D')+\bar{\nu}_{k-1}(D'')\leq {{(d-1)-2(k-1)}\choose{k}}+{{(d-3)-2(k-2)}\choose {k-1}}={{d-2(k-1)}\choose{k}}$. Since the number on the right equals $\nu_k(X_P)$ the thesis follows. 
\end{proof}

This argument allows to use induction on $d=2\delta+i$ and $k$ to prove that the degenerations depicted in Figure \ref{delightful rational}  of the trapezia $P^0(4,2\delta+i,\delta+i)$, $0\leq i \leq 3$, are $k$-delightful, for $k$ such that $3k-1\leq d+1$.

\begin{prop}
The triangulations in Figure \ref{delightful rational} are delightful.
\end{prop}
\begin{proof}
Let $D$ denote one of the triangulations of Figure \ref{delightful rational} and let $d=2\delta+i$ be the number of triangles of $D$. 

Fix $k=2$. We first prove that $D$ is $2$-delightful by induction on $d$ and exploiting the fact that $D$ is $2$-delightful if and only if $\bar{\nu}_2(D')={{d-3}\choose{2}}$ and $\bar{\nu}_1(D'')=\#(D'')=d-3$ (see the proof of Lemma \ref{D' del implica D del}), for each $D$ as in Figure \ref{delightful rational}. Then we consider the case $k\geq3$. 

For $d=4$, the degenerations $D_{2,2}$, $D'_{2,2}$ of $S(2,2)$ and $D'_{1,3}$ of $S(1,3)$ are clearly $1$-delightful indeed each of them contains exactly one pair of disjoint triangles and $\nu_2(S(2,2))=\nu_2(S(1,3))=1$.
The same holds in the case $d=5$ for $D_{2,3}$, $D'_{2,3}$ and $D'_{1,4}$: one can easily check that each contains exactly three pairs of disjoint triangles and it is $\nu_2(S(2,3))=\nu_2(S(1,4))=3$.

For $d\geq6$, assume the thesis true for any degree $\leq d-1$. If $d$ is even, write $d=2\delta$. The degeneration $D_{\delta,\delta}$ (or  $D'_{\delta,\delta}$) of $P^0(4,2\delta,\delta)$ is obtained from $D'=D_{\delta-1,\delta}$  ($D'=D'_{\delta-1,\delta}$ respectively)  by adding a triangle on the right. Now $\bar{\nu}_2(D)\leq\bar{\nu}_2(D')+\bar{\nu}_1(D'')={{d-2}\choose{2}}+d-3={{d-3}\choose{2}}=\nu_2(S(\delta,\delta))$. In the same way, the degeneration $D=D'_{\delta-1,(\delta-1)+2}$ of $P^0(4,2(\delta-1)+2,(\delta-1)+2)$ is obtained by adding a triangle to $D'_{\delta-1,(\delta-1)+1}$ and computing the number $\bar{\nu}_2(D)$ we get the same conclusion. 
If  $d$ is odd, write $d=2\delta+1$. The degenerations $D_{\delta,\delta+1}$ and $D'_{\delta,\delta+1}$ are obtained respectively from $D_{\delta,\delta}$ and $D'_{\delta-1,\delta+1}$ by adding a triangle on the right end and the computation done in the case $d$ even also works. Similarly, $D=D'_{\delta-1,(\delta-1)+3}$, which is obtained from $D'_{\delta-1,(\delta-1)+2}$, turns out to be $2$-delightful.

Now, fix $k\geq 3$, and consider $d$ such that $3k-1\leq d+1$. Let $D$ be one of the degenerations of $P=P^0(4,2\delta+i,\delta+i)$ in Figure \ref{delightful rational}. We prove the statement by induction on $d$ and $k$ using an argument similar to that of above. Let $D'$ and $D''$ be as above and assume $D'$ is $k$-delightful and $D''$ is $(k-1)$-delightful. Then $\bar{\nu}_{k}(D)=\bar{\nu}_{k}(D')+\bar{\nu}_{k-1}(D'')={{d-2(k-1)}\choose{k}}$.
\end{proof}

This proves Theorem \ref{delightful class g=0}.

Our result fits with the ones obtained by Sturmfels and Sullivant. In \cite[Prop. 5.8]{SS} they proved that if a delightful term order exists for a rational normal scroll $S(\delta_1,\dots,\delta_n$) of dimension $n$, then we must have $\delta_j\in\{m,m+1,m+2,m+3\}$ for some $m$. They also proved in \cite[Prop. 5.11]{SS} that the converse holds in the $n=2$ case. With our approach we have proved the same result in the case $n=2$ and we have also constructed these delightful triangulations.

\subsection{The elliptic case}
Here we prove a classification result for the $g=1$ case. The polytopes we are dealing with are depicted in Theorem \ref{census1}.
\begin{thm}\label{delightful class g=1}
All  polytopes with $g=1$ and $5\leq d\leq 8$ admit delightful triangulations. They are lattice equivalent to the ones in Figure \ref{delightful elliptic}. 
\begin{figure}[h!]
\begin{center}\begin{tabular}{ll}
\setlength{\unitlength}{0.333333mm}$l=3$: & 
\begin{picture}(60,45)(0,0)
\put(0,0){\line(0,1){40}}\put(0,20){\line(1,-1){20}}\put(0,40){\line(1,-1){40}}\put(0,40){\line(1,-2){20}}\put(0,40){\line(3,-2){60}}
\put(20,20){\line(2,-1){40}}\put(0,0){\line(1,0){60}}\put(20,0){\line(0,1){20}}
\end{picture}\ \ \ \ 
\begin{picture}(60,45)(0,0)
\put(0,0){\line(0,1){40}}\put(0,20){\line(1,-1){20}}\put(0,40){\line(1,-1){20}}\put(0,40){\line(3,-2){60}}
\put(0,20){\line(1,0){20}}\put(0,20){\line(2,-1){40}}
\put(20,20){\line(1,-1){20}}\put(20,20){\line(2,-1){40}}\put(0,0){\line(1,0){60}}
\end{picture}\ \ \ \ 
\setlength{\unitlength}{0.333333mm}
\begin{picture}(80,45)(0,0)
\put(0,0){\line(0,1){40}}\put(0,20){\line(1,-1){20}}\put(0,40){\line(1,-1){20}}\put(0,40){\line(3,-2){60}}
\put(0,20){\line(1,0){20}}\put(0,20){\line(2,-1){40}}
\put(20,20){\line(1,-1){20}}\put(20,20){\line(2,-1){40}}\put(0,0){\line(1,0){80}}
\put(0,40){\line(2,-1){80}}\put(40,20){\line(1,-1){20}}
\end{picture}
\\
$l=4$: &
\begin{picture}(40,45)(0,0)
\put(0,0){\line(0,1){40}}\put(0,20){\line(1,1){20}}\put(0,40){\line(1,0){20}}\put(0,0){\line(1,1){20}}
\put(0,0){\line(1,2){20}}\put(0,0){\line(1,0){40}}\put(20,0){\line(0,1){40}}
\put(20,20){\line(1,-1){20}}\put(20,40){\line(1,-2){20}}
\end{picture}\ \ \ \ 
\begin{picture}(40,45)(0,0)
\put(0,0){\line(0,1){40}}\put(0,20){\line(1,-1){20}}\put(0,40){\line(1,-2){20}}\put(0,40){\line(1,0){20}}\put(0,40){\line(1,-1){40}}
\put(0,0){\line(1,0){40}}
\put(20,0){\line(0,1){40}}\put(20,40){\line(1,-2){20}}
\end{picture}\ \ \ \ 
\begin{picture}(40,45)(0,0)
\put(0,0){\line(0,1){40}}\put(0,0){\line(1,0){40}}\put(0,20){\line(1,-1){20}}\put(0,20){\line(1,0){20}}\put(0,40){\line(1,-1){40}}\put(0,20){\line(2,-1){40}}\put(40,0){\line(-1,2){20}}\put(0,40){\line(1,0){20}}\put(20,20){\line(0,1){20}}
\end{picture}
\ \ \ \ 
\setlength{\unitlength}{0.333333mm}
\begin{picture}(40,45)(0,0)
\put(0,0){\line(0,1){40}}\put(0,20){\line(1,-1){20}}\put(0,40){\line(1,-1){40}}\put(0,40){\line(1,-2){20}}
\put(0,0){\line(1,0){40}}\put(20,0){\line(0,1){40}}
\put(20,40){\line(1,-1){20}}\put(20,40){\line(1,-2){20}}\put(40,0){\line(0,1){40}}\put(0,40){\line(1,0){40}}
\end{picture}\ \ \ \ 
\begin{picture}(60,45)(0,0)
\put(0,0){\line(0,1){20}}\put(0,0){\line(1,0){60}}\put(0,20){\line(1,1){20}}\put(20,40){\line(1,-1){40}}\put(0,20){\line(1,0){20}}\put(0,20){\line(1,-1){20}}\put(20,0){\line(0,1){40}}\put(20,20){\line(1,-1){20}}\put(20,40){\line(1,-2){20}}\put(40,0){\line(0,1){20}}
\end{picture}\ \ \ \ 
\begin{picture}(60,45)(0,0)
\put(0,0){\line(0,1){20}}\put(0,0){\line(1,0){60}}\put(0,20){\line(1,1){20}}\put(20,40){\line(1,-1){40}}\put(0,0){\line(1,2){20}}\put(0,0){\line(1,1){20}}\put(20,0){\line(0,1){40}}\put(20,20){\line(1,0){20}}\put(20,0){\line(1,1){20}}\put(40,0){\line(0,1){20}}
\end{picture} \\ \ & 
\begin{picture}(40,45)(0,0)
\put(0,0){\line(0,1){20}}\put(0,0){\line(1,0){40}}\put(0,20){\line(1,1){20}}\put(20,40){\line(1,-2){20}}\put(0,0){\line(1,2){20}}\put(0,0){\line(1,1){20}}\put(20,0){\line(0,1){40}}\put(20,20){\line(1,-1){20}}
\end{picture}\ \ \ \ 
\begin{picture}(40,45)(0,0)
\put(0,0){\line(0,1){20}}\put(0,0){\line(1,0){40}}\put(0,20){\line(1,1){20}}\put(20,40){\line(1,-2){20}}
\put(0,20){\line(1,-1){20}}\put(0,20){\line(1,0){20}}\put(20,0){\line(0,1){40}}\put(20,20){\line(1,-1){20}}
\end{picture}\ \ \ \ 
\begin{picture}(40,45)(0,0)
\put(0,0){\line(0,1){20}}\put(0,0){\line(1,0){40}}\put(0,20){\line(1,1){20}}\put(20,40){\line(1,-2){20}}
\put(0,20){\line(1,-1){20}}\put(0,20){\line(2,-1){40}}\put(0,20){\line(1,0){20}}\put(20,20){\line(0,1){20}}\put(20,20){\line(1,-1){20}}
\end{picture}\ \ \ \ 
\begin{picture}(60,45)(0,0)
\put(0,0){\line(0,1){40}}\put(0,20){\line(1,-1){20}}\put(0,40){\line(1,-1){40}}\put(0,40){\line(1,-2){20}}
\put(0,0){\line(1,0){60}}\put(20,0){\line(0,11){40}}\put(40,20){\line(1,-1){20}}
\put(20,40){\line(1,-1){20}}\put(20,40){\line(1,-2){20}}\put(40,0){\line(0,1){20}}\put(0,40){\line(1,0){20}}
\end{picture}\ \ \ \ 
\begin{picture}(60,45)(0,0)
\put(0,0){\line(0,1){40}}\put(0,0){\line(1,0){60}}\put(0,40){\line(1,0){20}}\put(20,40){\line(1,-1){40}}
\put(0,0){\line(1,1){20}}\put(0,00){\line(1,2){20}}\put(0,20){\line(1,1){20}}\put(20,0){\line(0,1){40}}\put(20,0){\line(1,1){20}}\put(20,20){\line(1,0){20}}\put(40,0){\line(0,1){20}}
\end{picture}\\
\setlength{\unitlength}{0.333333mm}$l=5$: &
\begin{picture}(40,45)(0,0)
\put(0,0){\line(0,1){20}}\put(0,20){\line(1,0){40}}
\put(0,0){\line(1,0){20}}\put(20,0){\line(0,1){40}}
\put(20,0){\line(1,1){20}}\put(20,40){\line(1,-1){20}}\put(0,20){\line(1,-1){20}}\put(0,20){\line(1,1){20}}
\end{picture}\ \ \ \ 
\begin{picture}(40,45)(0,0)
\put(0,0){\line(0,1){20}}\put(0,20){\line(1,0){40}}\put(20,20){\line(0,1){20}}
\put(0,0){\line(1,0){20}}\put(0,0){\line(1,1){20}}\put(0,20){\line(1,1){20}}
\put(20,0){\line(1,1){20}}\put(20,40){\line(1,-1){20}}\put(0,0){\line(2,1){40}}
\end{picture}\ \ \ \ 
\begin{picture}(40,45)(0,0)
\put(0,0){\line(0,1){20}}\put(0,20){\line(1,0){40}}\put(20,20){\line(0,1){20}}
\put(0,0){\line(1,0){20}}\put(0,0){\line(1,1){20}}\put(0,20){\line(1,1){20}}
\put(20,0){\line(1,1){20}}\put(20,40){\line(1,-1){20}}\put(0,0){\line(2,1){40}}
\put(20,0){\line(1,0){20}}\put(40,0){\line(0,1){20}}
\end{picture}\ \ \ \ 
\begin{picture}(40,45)(0,0)
\put(0,0){\line(0,1){20}}\put(0,0){\line(1,0){40}}\put(0,20){\line(1,1){20}}\put(0,20){\line(1,0){20}}\put(20,0){\line(0,1){40}}
\put(20,20){\line(1,-1){20}}\put(20,40){\line(1,-1){20}}\put(20,40){\line(1,-2){20}}\put(40,0){\line(0,1){20}}\put(0,20){\line(1,-1){20}}
\end{picture}\ \ \ \ 
\begin{picture}(40,45)(0,0)
\put(0,0){\line(0,1){40}}\put(0,0){\line(1,0){40}}\put(0,40){\line(1,0){20}}\put(20,40){\line(1,-1){20}}
\put(0,0){\line(1,1){20}}\put(0,0){\line(1,2){20}}\put(0,20){\line(1,1){20}}\put(20,0){\line(0,1){40}}\put(20,0){\line(1,1){20}}\put(20,20){\line(1,0){20}}\put(40,0){\line(0,1){20}}
\end{picture}\ \ \ \ 
\begin{picture}(40,45)(0,0)
\put(0,0){\line(0,1){40}}\put(0,20){\line(1,-1){20}}\put(0,40){\line(1,-1){40}}\put(0,40){\line(1,-2){20}}
\put(0,0){\line(1,0){40}}\put(20,0){\line(0,11){40}}
\put(20,40){\line(1,-1){20}}\put(20,40){\line(1,-2){20}}\put(40,0){\line(0,1){20}}\put(0,40){\line(1,0){20}}
\end{picture}\\
\setlength{\unitlength}{0.333333mm}$l=6$: &
\begin{picture}(40,45)(0,0)
\put(0,0){\line(0,1){20}}\put(0,20){\line(1,0){40}}
\put(0,0){\line(1,0){20}}\put(20,0){\line(0,1){40}}\put(20,40){\line(1,0){20}}\put(40,20){\line(0,1){20}}
\put(20,0){\line(1,1){20}}\put(20,40){\line(1,-1){20}}\put(0,20){\line(1,-1){20}}\put(0,20){\line(1,1){20}}
\end{picture}
\end{tabular}\end{center}
\caption{Delightful triangulations of $g=1$ polytopes}\label{delightful elliptic}
\end{figure}
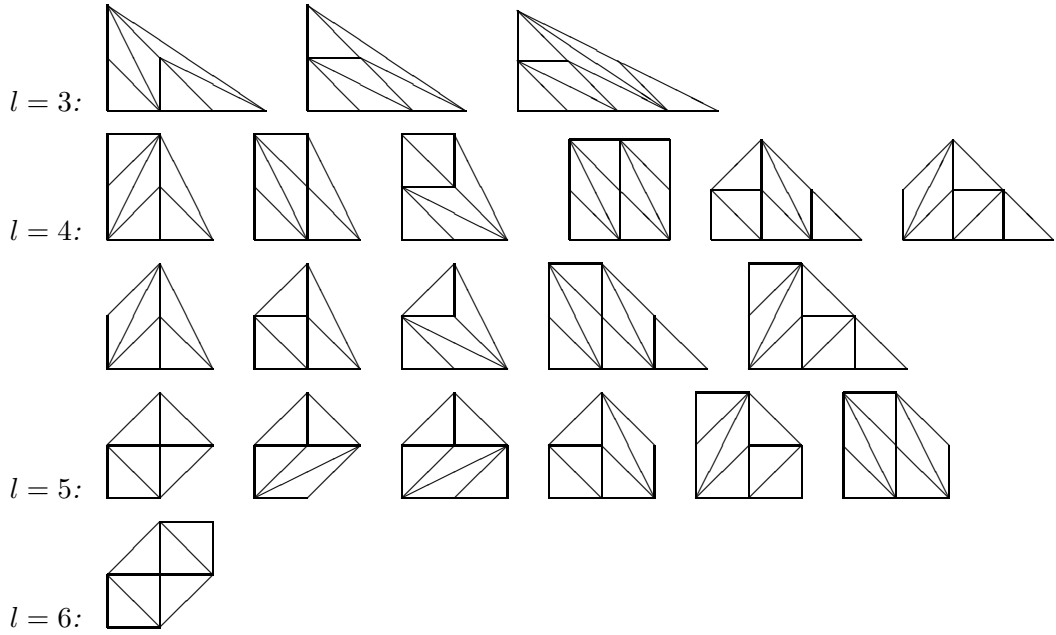
\end{thm}

 We need a preliminary remark. Notice that if $D$ do not contain an elliptic singularity of multiplicity $d$, then $D$ must contain at least a triangle $T$ that do not have any vertex at the interior lattice point and  $P\setminus T$ is convex.
 Define $P':=P\setminus T\subseteq P$: $\textrm{Area}(P')+1=\textrm{Area}(P)=d$, $g(P)=g(P')=1$ and assume that that $\dim(\Sec_2(X_P))=\dim(\Sec_2(X_{P'}))=5$. Consider the triangulation $D'\subseteq D$ of $P'$ obtained from $D$ by deleting that $T$,  $D'$ is still regular. 
Set $D''=\{T''\in D': T''\cap T=\emptyset\}\subseteq D'$: we have that $\#(D'')\leq d-4$. From this follows that   $\bar{\nu}_{2}(D)=\bar{\nu}_{2}(D')+\#(D'')\leq {{d-4}\choose 2}+(d-4)={{d-3}\choose 2}=\nu_2(X_P)$. 
We get the following lemma.

\begin{lemma}\label{lemma g=1 delightful}
In the notation of above, $D$ is $2$-delightful if and only if there exists a triangle $T$ such that $D'=D\setminus T$ is $2$-delightful and such that there are exactly $d-4$ triangles in $D$ not intersecting $T$.
\end{lemma}

\begin{proof}[Proof of Theorem \ref{delightful class g=1}]
Assume $k=2$. We start from the base case, $d=5$ and then we increase the degree by adding a triangle. In this way we can exploit Lemma \ref{lemma g=1 delightful} and cover all cases $5\leq d\leq9$.

We first consider the sub-polytopes of the triangle $P^1(3,9,3)$, i.e. the ones corresponding to internal projections  of the $3$-ple Veronese embedding of $\p^2$ in $\p^9$. 
 
Fix $d=5$. 
There are finitely many (regular) triangulations of each of these polytopes, up to equiaffinity.  A part from the cases with five triangles covering the interior lattice point (see Table \ref{caso ellittico sec 1}, first row), and from the case
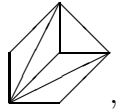
\begin{center}
\setlength{\unitlength}{0.333333mm}
\begin{picture}(40,40)(0,0)
\put(0,0){\line(0,1){20}}\put(0,0){\line(1,0){20}}\put(20,0){\line(1,1){20}}\put(20,40){\line(1,-1){20}}\put(0,20){\line(1,1){20}}
\put(0,0){\line(1,1){20}}\put(0,0){\line(2,1){40}}\put(0,0){\line(1,2){20}}\put(20,20){\line(0,1){20}}\put(20,20){\line(1,0){20}}
\end{picture},
\end{center}
the remaining configurations contain a unique skew $2$-set, so they are $2$-delightful. It is easy and left to the reader.

Now fix $d=6$.
The only possible way to get $2$-delightful triangulations of $P^1(4,6,2)$  is adding a triangle $T$ to the $2$-delightful triangulations of subpolytopes with $d=5$ such that there are $2$ triangles in $D$ not intersecting $T$, by Lemma \ref{lemma g=1 delightful}.
The candidates have to be chosen among the following configurations
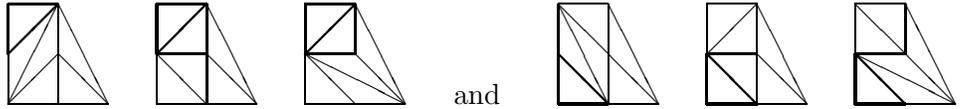
\begin{center}
\setlength{\unitlength}{0.333333mm}
\begin{picture}(40,40)(0,0)
\put(0,0){\line(0,1){20}}\put(0,0){\line(1,0){40}}\put(0,20){\line(1,1){20}}\put(20,40){\line(1,-2){20}}\put(0,0){\line(1,2){20}}\put(0,0){\line(1,1){20}}\put(20,0){\line(0,1){40}}\put(20,20){\line(1,-1){20}} 
\put(0,20){\thicklines\line(0,1){20}}\put(0,40){\thicklines\line(1,0){20}}\put(0,20){\thicklines\line(1,1){20}}
\end{picture}\ \ \ \ \ 
\begin{picture}(40,40)(0,0)
\put(0,0){\line(0,1){20}}\put(0,0){\line(1,0){40}}\put(0,20){\line(1,1){20}}\put(20,40){\line(1,-2){20}}
\put(0,20){\line(1,-1){20}}\put(0,20){\line(1,0){20}}\put(20,0){\line(0,1){40}}\put(20,20){\line(1,-1){20}} 
\put(0,20){\thicklines\line(0,1){20}}\put(0,40){\thicklines\line(1,0){20}}\put(0,20){\thicklines\line(1,1){20}}
\end{picture}\ \ \ \ \ 
\begin{picture}(40,40)(0,0)
\put(0,0){\line(0,1){20}}\put(0,0){\line(1,0){40}}\put(0,20){\line(1,1){20}}\put(20,40){\line(1,-2){20}}
\put(0,20){\line(1,-1){20}}\put(0,20){\line(2,-1){40}}\put(0,20){\line(1,0){20}}\put(20,20){\line(0,1){20}}\put(20,20){\line(1,-1){20}}
\put(0,20){\thicklines\line(0,1){20}}\put(0,40){\thicklines\line(1,0){20}}\put(0,20){\thicklines\line(1,1){20}}
\end{picture}\ \ \ \ \ and \ \ \ \ \ 
\begin{picture}(40,40)(0,0)
\put(0,0){\line(0,1){40}}\put(0,20){\line(1,-1){20}}\put(0,40){\line(1,-2){20}}\put(0,40){\line(1,0){20}}\put(0,40){\line(1,-1){40}}
\put(0,0){\line(1,0){40}}
\put(20,0){\line(0,1){40}}\put(20,40){\line(1,-2){20}}
\put(0,0){\thicklines\line(1,0){20}}\put(0,0){\thicklines\line(0,1){20}}\put(0,20){\thicklines\line(1,-1){20}}
\end{picture}\ \ \ \ \ 
\begin{picture}(40,40)(0,0)
\put(0,0){\line(0,1){40}}\put(0,20){\line(1,1){20}}\put(0,40){\line(1,0){20}}\put(0,20){\line(1,-1){20}}
\put(0,0){\line(1,0){40}}\put(20,0){\line(0,1){40}}\put(20,20){\line(1,-1){20}}\put(20,40){\line(1,-2){20}}\put(0,20){\line(1,0){20}}
\put(0,0){\thicklines\line(1,0){20}}\put(0,0){\thicklines\line(0,1){20}}\put(0,20){\thicklines\line(1,-1){20}}
\end{picture}\ \ \ \ \ 
\begin{picture}(40,40)(0,0)
\put(0,0){\line(0,1){40}}\put(0,0){\line(1,0){40}}\put(0,20){\line(1,-1){20}}\put(0,20){\line(1,0){20}}\put(0,40){\line(1,-1){40}}\put(0,20){\line(2,-1){40}}\put(40,0){\line(-1,2){20}}\put(0,40){\line(1,0){20}}\put(20,20){\line(0,1){20}}
\put(0,0){\thicklines\line(1,0){20}}\put(0,0){\thicklines\line(0,1){20}}\put(0,20){\thicklines\line(1,-1){20}}
\end{picture}\ .
\end{center}
The first set was obtained by adding a triangle to the $2$-delightful  triangulations of the quadrilateral $P^1(4,5,2)$ in Figure \ref{delightful elliptic}. Instead, to get the second set of configurations we first chose representatives of the equiaffinity class of the  $2$-delightful triangulations of $P^1(5,5,1)$ depicted in Figure \ref{delightful elliptic}, namely 
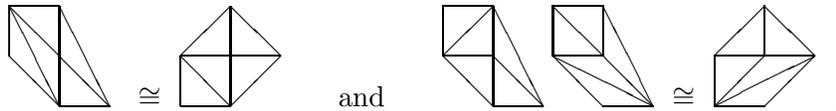
\begin{center}
\setlength{\unitlength}{0.333333mm} 
\begin{picture}(40,40)(0,0)
\put(0,20){\line(0,1){20}}\put(0,20){\line(1,-1){20}}\put(0,40){\line(1,-2){20}}\put(0,40){\line(1,0){20}}\put(0,40){\line(1,-1){40}}
\put(20,0){\line(1,0){20}}
\put(20,0){\line(0,1){40}}\put(20,40){\line(1,-2){20}}
\end{picture} \ \ $\cong$\ \ 
\begin{picture}(40,45)(0,0)
\put(0,0){\line(0,1){20}}\put(0,20){\line(1,0){40}}
\put(0,0){\line(1,0){20}}\put(20,0){\line(0,1){40}}
\put(20,0){\line(1,1){20}}\put(20,40){\line(1,-1){20}}\put(0,20){\line(1,-1){20}}\put(0,20){\line(1,1){20}}
\end{picture} \ \ \ \ \   and  \ \ \ \ \ 
\begin{picture}(40,40)(0,0)
\put(0,20){\line(0,1){20}}\put(0,20){\line(1,1){20}}\put(0,40){\line(1,0){20}}\put(0,20){\line(1,-1){20}}
\put(20,0){\line(1,0){20}}\put(20,0){\line(0,1){40}}\put(20,20){\line(1,-1){20}}\put(20,40){\line(1,-2){20}}\put(0,20){\line(1,0){20}}
\end{picture}\ 
\begin{picture}(40,40)(0,0)
\put(0,20){\line(0,1){20}}\put(20,0){\line(1,0){20}}\put(0,20){\line(1,-1){20}}\put(0,20){\line(1,0){20}}\put(0,40){\line(1,-1){40}}\put(0,20){\line(2,-1){40}}\put(40,0){\line(-1,2){20}}\put(0,40){\line(1,0){20}}\put(20,20){\line(0,1){20}}
\end{picture}\ \  $\cong$\ \ 
\begin{picture}(40,45)(0,0)
\put(0,0){\line(0,1){20}}\put(0,20){\line(1,0){40}}\put(20,20){\line(0,1){20}}
\put(0,0){\line(1,0){20}}\put(0,0){\line(1,1){20}}\put(0,20){\line(1,1){20}}
\put(20,0){\line(1,1){20}}\put(20,40){\line(1,-1){20}}\put(0,0){\line(2,1){40}}
\end{picture} \ ,
\end{center}
and then we added a triangle, as shown in the pictures. There are two triangles not intersecting $T$ in  the first, the forth and the sixth triangulation of $P(5,6,2)$ depicted above, so in these cases there are in all three skew $2$-sets and we have $2$-delightfulness.

Similarly, for $P¹(5,6,2)$ we may choose among the following configurations
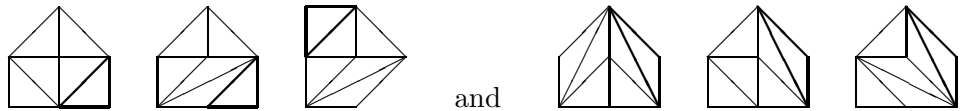
\begin{center}
\setlength{\unitlength}{0.333333mm}
\begin{picture}(40,40)(0,0)
\put(0,0){\line(0,1){20}}\put(0,20){\line(1,0){40}}\put(0,0){\line(1,0){20}}\put(20,0){\line(0,1){40}}
\put(20,0){\line(1,1){20}}\put(20,40){\line(1,-1){20}}\put(0,20){\line(1,-1){20}}\put(0,20){\line(1,1){20}}
\put(20,0){\thicklines\line(1,1){20}}\put(40,0){\thicklines\line(0,1){20}}\put(20,0){\thicklines\line(1,0){20}}
\end{picture}\ \ \ \ \ 
\begin{picture}(40,40)(0,0)
\put(0,0){\line(0,1){20}}\put(0,20){\line(1,0){40}}\put(20,20){\line(0,1){20}}\put(0,0){\line(1,0){20}}\put(0,0){\line(1,1){20}}\put(0,20){\line(1,1){20}}\put(20,0){\line(1,1){20}}\put(20,40){\line(1,-1){20}}\put(0,0){\line(2,1){40}}\put(20,0){\thicklines\line(1,1){20}}\put(40,0){\thicklines\line(0,1){20}}\put(20,0){\thicklines\line(1,0){20}}
\end{picture}\ \ \ \ \ 
\begin{picture}(40,40)(0,0)
\put(0,0){\line(0,1){20}}\put(0,20){\line(1,0){40}}\put(20,20){\line(0,1){20}}\put(0,0){\line(1,0){20}}\put(0,0){\line(1,1){20}}\put(0,20){\line(1,1){20}}\put(20,0){\line(1,1){20}}\put(20,40){\line(1,-1){20}}\put(0,0){\line(2,1){40}}\put(0,20){\thicklines\line(1,1){20}}\put(0,20){\thicklines\line(0,1){20}}\put(0,40){\thicklines\line(1,0){20}}
\end{picture}
\ \ \ \ \  and \ \ \ \ \ 
\begin{picture}(40,40)(0,0)
\put(0,0){\line(0,1){20}}\put(0,0){\line(1,0){40}}\put(0,20){\line(1,1){20}}\put(20,40){\line(1,-2){20}}\put(0,0){\line(1,2){20}}\put(0,0){\line(1,1){20}}\put(20,0){\line(0,1){40}}\put(20,20){\line(1,-1){20}}
\put(20,40){\thicklines\line(1,-2){20}}\put(40,0){\thicklines\line(0,1){20}}\put(20,40){\thicklines\line(1,-1){20}}
\end{picture}\ \ \ \ \ 
\begin{picture}(40,40)(0,0)
\put(0,0){\line(0,1){20}}\put(0,0){\line(1,0){40}}\put(0,20){\line(1,1){20}}\put(20,40){\line(1,-2){20}}\put(0,20){\line(1,-1){20}}\put(0,20){\line(1,0){20}}\put(20,0){\line(0,1){40}}\put(20,20){\line(1,-1){20}}
\put(20,40){\thicklines\line(1,-2){20}}\put(40,0){\thicklines\line(0,1){20}}\put(20,40){\thicklines\line(1,-1){20}}
\end{picture}\ \ \ \ \ 
\begin{picture}(40,40)(0,0)
\put(0,0){\line(0,1){20}}\put(0,0){\line(1,0){40}}\put(0,20){\line(1,1){20}}\put(20,40){\line(1,-2){20}}
\put(0,20){\line(1,-1){20}}\put(0,20){\line(2,-1){40}}\put(0,20){\line(1,0){20}}\put(20,20){\line(0,1){20}}\put(20,20){\line(1,-1){20}}
\put(20,40){\thicklines\line(1,-2){20}}\put(40,0){\thicklines\line(0,1){20}}\put(20,40){\thicklines\line(1,-1){20}}
\end{picture}
\end{center}
The second triangulation is $2$-delightful and the same holds for the third and the fifth which are lattice equivalent. While the remaining configurations contain less than $3$ pairs of disjoint triangles.

For the hexagon $P^1(6,6,1)$, the candidates are 
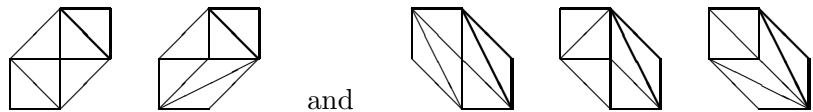
\begin{center}
\setlength{\unitlength}{0.333333mm}
\begin{picture}(40,40)(0,0)
\put(0,0){\line(0,1){20}}\put(0,20){\line(1,0){40}}\put(0,0){\line(1,0){20}}\put(20,0){\line(0,1){40}}
\put(20,0){\line(1,1){20}}\put(20,40){\line(1,-1){20}}\put(0,20){\line(1,-1){20}}\put(0,20){\line(1,1){20}}
\put(20,40){\thicklines\line(1,0){20}}\put(40,20){\thicklines\line(0,1){20}}\put(20,40){\thicklines\line(1,-1){20}}
\end{picture}\ \ \ \ \ 
\begin{picture}(40,40)(0,0)
\put(0,0){\line(0,1){20}}\put(0,20){\line(1,0){40}}\put(20,20){\line(0,1){20}}\put(0,0){\line(1,0){20}}\put(0,0){\line(1,1){20}}\put(0,20){\line(1,1){20}}\put(20,0){\line(1,1){20}}\put(20,40){\line(1,-1){20}}\put(0,0){\line(2,1){40}}
\put(20,40){\thicklines\line(1,0){20}}\put(40,20){\thicklines\line(0,1){20}}\put(20,40){\thicklines\line(1,-1){20}}
\end{picture}
\ \ \ \ \ and \ \ \ \ \ 
\begin{picture}(40,40)(0,0)
\put(0,20){\line(0,1){20}}\put(0,20){\line(1,-1){20}}\put(0,40){\line(1,-2){20}}\put(0,40){\line(1,0){20}}\put(0,40){\line(1,-1){40}}
\put(20,0){\line(1,0){20}}\put(20,0){\line(0,1){40}}\put(20,40){\line(1,-2){20}}
\put(20,40){\thicklines\line(1,-2){20}}\put(40,0){\thicklines\line(0,1){20}}\put(20,40){\thicklines\line(1,-1){20}}
\end{picture}\ \ \ \ \ 
\begin{picture}(40,40)(0,0)
\put(0,20){\line(0,1){20}}\put(0,20){\line(1,1){20}}\put(0,40){\line(1,0){20}}\put(0,20){\line(1,-1){20}}
\put(20,0){\line(1,0){20}}\put(20,0){\line(0,1){40}}\put(20,20){\line(1,-1){20}}\put(20,40){\line(1,-2){20}}\put(0,20){\line(1,0){20}}
\put(20,40){\thicklines\line(1,-2){20}}\put(40,0){\thicklines\line(0,1){20}}\put(20,40){\thicklines\line(1,-1){20}}
\end{picture}\ \ \ \ \ 
\begin{picture}(40,40)(0,0)
\put(0,20){\line(0,1){20}}\put(20,0){\line(1,0){20}}\put(0,20){\line(1,-1){20}}\put(0,20){\line(1,0){20}}\put(0,40){\line(1,-1){40}}\put(0,20){\line(2,-1){40}}\put(40,0){\line(-1,2){20}}\put(0,40){\line(1,0){20}}\put(20,20){\line(0,1){20}}
\put(20,40){\thicklines\line(1,-2){20}}\put(40,0){\thicklines\line(0,1){20}}\put(20,40){\thicklines\line(1,-1){20}}
\end{picture}
\end{center}
Just the first and the third, which are lattice equivalent, are $2$-delightful.

For $d=7,8$, namely for the polytopes $P^1(4,7,3)$, $P^1(5,7,2)$ and $P^1(4,8,3)$, the proof is similar and the details are left to the reader. 

Consider finally the triangle $P^1(3,9,3)$. If there was any $2$-delightful triangulation, it would be obtained by adding a triangle to some $2$-delightful triangulation of $P^1(4,8,3)$,  i.e.,
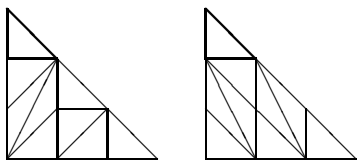
\begin{center}
\setlength{\unitlength}{0.333333mm}
\begin{picture}(60,60)(0,0)
\put(0,0){\line(0,1){60}}\put(0,60){\line(1,-1){20}}\put(0,0){\line(1,0){60}}\put(0,40){\line(1,0){20}}\put(20,40){\line(1,-1){20}}
\put(0,0){\line(1,1){20}}\put(0,0){\line(1,2){20}}\put(0,20){\line(1,1){20}}\put(20,0){\line(0,1){40}}\put(20,0){\line(1,1){20}}\put(20,20){\line(1,0){20}}\put(40,0){\line(0,1){20}}\put(40,20){\line(1,-1){20}}
\put(0,60){\thicklines\line(1,-1){20}}\put(0,40){\thicklines\line(1,0){20}}\put(0,40){\thicklines\line(0,1){20}}
\end{picture}\ \ \ \ \ 
\begin{picture}(60,60)(0,0)
\put(0,0){\line(0,1){60}}\put(0,60){\line(1,-1){20}}\put(0,20){\line(1,-1){20}}\put(0,40){\line(1,-1){40}}\put(0,40){\line(1,-2){20}}
\put(0,0){\line(1,0){60}}\put(20,0){\line(0,11){40}}\put(40,20){\line(1,-1){20}}
\put(20,40){\line(1,-1){20}}\put(20,40){\line(1,-2){20}}\put(40,0){\line(0,1){20}}\put(0,40){\line(1,0){20}}
\put(0,60){\thicklines\line(1,-1){20}}\put(0,40){\thicklines\line(1,0){20}}\put(0,40){\thicklines\line(0,1){20}}
\end{picture}
\end{center}
but they both are not $2$-delightful.

Now we  analyze the remaining polytopes.
 To get possible $2$-delightful triangulations of $P^1(4,8,2)$ we add a triangle to the $2$-delightful triangulations of $P^1(5,7,2)$ in Figure \ref{delightful elliptic}:
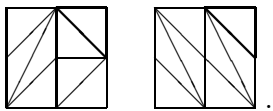
\begin{center}
\setlength{\unitlength}{0.333333mm}
\begin{picture}(40,40)(0,0)
\put(0,0){\line(0,1){40}}\put(0,0){\line(1,0){40}}\put(0,40){\line(1,0){20}}\put(20,40){\line(1,-1){20}}\put(0,0){\line(1,1){20}}\put(0,0){\line(1,2){20}}\put(0,20){\line(1,1){20}}\put(20,0){\line(0,1){40}}\put(20,0){\line(1,1){20}}\put(20,20){\line(1,0){20}}\put(40,0){\line(0,1){20}}
\put(20,40){\thicklines\line(1,-1){20}}\put(40,20){\thicklines\line(0,1){20}}\put(20,40){\thicklines\line(1,0){20}}
\end{picture}\ \ \ \ \ 
\begin{picture}(40,40)(0,0)
\put(0,0){\line(0,1){40}}\put(0,20){\line(1,-1){20}}\put(0,40){\line(1,-1){40}}\put(0,40){\line(1,-2){20}}\put(0,0){\line(1,0){40}}\put(20,0){\line(0,11){40}}\put(20,40){\line(1,-1){20}}\put(20,40){\line(1,-2){20}}\put(40,0){\line(0,1){20}}\put(0,40){\line(1,0){20}}
\put(20,40){\thicklines\line(1,-1){20}}\put(40,20){\thicklines\line(0,1){20}}\put(20,40){\thicklines\line(1,0){20}}
\end{picture}\ .
\end{center}
The second  is the unique that is $2$-delightful since  $T$ is disjoint from four distinct triangles and we have six more skew $2$-sets from the subdivision of $P^1(5,7,2)$. So $\bar{\nu}_2(D)=10$.

Consider $P^1(3,6,3)$. The candidates are the configurations obtained by adding a triangle to the $2$-delightful triangulations of $P^1(4,5,2)$. We get:
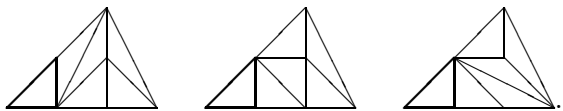
\begin{center}
\setlength{\unitlength}{0.333333mm}
\begin{picture}(60,40)(-20,0)
\put(0,0){\line(0,1){20}}\put(0,0){\line(1,0){40}}\put(0,20){\line(1,1){20}}\put(20,40){\line(1,-2){20}}\put(0,0){\line(1,2){20}}\put(0,0){\line(1,1){20}}\put(20,0){\line(0,1){40}}\put(20,20){\line(1,-1){20}}
\put(-20,0){\thicklines\line(1,1){20}}\put(-20,0){\thicklines\line(1,0){20}}\put(00,0){\thicklines\line(0,1){20}}
\end{picture}\ \ \ \ \ 
\begin{picture}(60,40)(-20,0)
\put(0,0){\line(0,1){20}}\put(0,0){\line(1,0){40}}\put(0,20){\line(1,1){20}}\put(20,40){\line(1,-2){20}}
\put(0,20){\line(1,-1){20}}\put(0,20){\line(1,0){20}}\put(20,0){\line(0,1){40}}\put(20,20){\line(1,-1){20}}
\put(-20,0){\thicklines\line(1,1){20}}\put(-20,0){\thicklines\line(1,0){20}}\put(00,0){\thicklines\line(0,1){20}}
\end{picture}\ \ \ \ \ 
\begin{picture}(60,40)(-20,0)
\put(0,0){\line(0,1){20}}\put(0,0){\line(1,0){40}}\put(0,20){\line(1,1){20}}\put(20,40){\line(1,-2){20}}
\put(0,20){\line(1,-1){20}}\put(0,20){\line(2,-1){40}}\put(0,20){\line(1,0){20}}\put(20,20){\line(0,1){20}}\put(20,20){\line(1,-1){20}}
\put(-20,0){\thicklines\line(1,1){20}}\put(-20,0){\thicklines\line(1,0){20}}\put(00,0){\thicklines\line(0,1){20}}
\end{picture}.
\end{center}
The first and the second are $2$-delightful, since we have two skew $2$-sets coming from the triangulation of $P^1(4,5,2)$ and one more pair involving $T$. 

Finally, for $P^1(3,8,4)$ the proof is similar and left to the reader.

Now assume $k\geq3$. If $P=P^1(4,8,3)$, then $\Sec_3(X_P)$ fills up the space $\p^8$. The subdivisions of $P$ depicted in Figure \ref{delightful elliptic}, that are $2$-delightful, are also $3$-delightful since a (unique) skew $3$-set exists in both of them. In all other cases, namely for $P$ not belonging to the class $P=P^1(4,8,3)$ and $D$ as in Figure \ref{delightful elliptic}, we have  $\dim(\Sec_3(X))<8$, see Section \ref{secant degree}.
\end{proof}


\section{Tables}\label{tables}
In the following tables, we collect the triangulations of polytopes with $g\leq 1$, in which all the triangles have a common vertex $p$. They are non-delightful and in particular correspond to the singularities that cause $k$-delightfulness defect, for $k=2,3$, see Theorem \ref{general sum}.

In the first column we draw the subdivision of the polytope $Q=Q_p$; the degree of $Z=Z_{p}$, which corresponds to the number of triangles, is written in the second column, while the numbers $\nu_2(Z)$ and $\nu_3(Z)$ are collected respectively in the third and in the fourth column.
 
\newpage

\begin{figure}[h!]
\begin{tabular}{|l|l|c|c|c|}
 \hline 
\ & \ \ triangulation of $Q$\ \  & $\deg(Z)$ & $\nu_2(Z)$  & $\nu_3(Z)$\\
\hline
\hline
1.&
\setlength{\unitlength}{0.333333mm}
\begin{picture}(60,25)(0,0)
\put(0,0){\line(0,1){20}}
\put(0,0){\line(1,0){60}}
\put(0,20){\line(1,0){20}}
\put(0,0){\line(1,1){20}}
\put(20,0){\line(0,1){20}}
\put(20,20){\line(1,-1){20}}
\put(20,20){\line(2,-1){40}}
\put(17,16){$\bullet$}
\end{picture}\ \ \ \ \ 
\begin{picture}(60,25)(0,0)
\put(0,0){\line(0,1){20}}
\put(0,0){\line(1,0){60}}
\put(0,20){\line(1,0){20}}
\put(0,20){\line(1,-1){20}}\put(0,20){\line(2,-1){40}}\put(0,20){\line(3,-1){60}}
\put(20,20){\line(2,-1){40}}
\put(-3,16){$\bullet$}
\end{picture}
 & 4 & 1 & /\\
 \hline
 2.&
\setlength{\unitlength}{0.333333mm}
\begin{picture}(60,25)(0,0)
\put(0,0){\line(0,1){20}}
\put(0,0){\line(1,0){40}}
\put(0,20){\line(1,0){40}}
\put(40,0){\line(0,1){20}}
\put(20,0){\line(0,1){20}}
\put(0,0){\line(1,1){20}}
\put(20,20){\line(1,-1){20}}
\put(17,16){$\bullet$}
\end{picture} & 4 & 1 & /\\
 \hline
 3.& 
\setlength{\unitlength}{0.333333mm}
\begin{picture}(80,25)(0,0)
\put(0,0){\line(0,1){20}}
\put(0,0){\line(1,0){80}}
\put(0,20){\line(1,0){20}}
\put(0,0){\line(1,1){20}}
\put(20,0){\line(0,1){20}}
\put(20,20){\line(1,-1){20}}
\put(20,20){\line(2,-1){40}}
\put(20,20){\line(3,-1){60}}
\put(17,16){$\bullet$}
\end{picture}\ \ \ \ \ 
\begin{picture}(80,25)(0,0)
\put(0,0){\line(0,1){20}}
\put(0,0){\line(1,0){80}}
\put(0,20){\line(1,0){20}}
\put(0,20){\line(1,-1){20}}
\put(0,20){\line(2,-1){40}}
\put(0,20){\line(3,-1){60}}
\put(0,20){\line(4,-1){80}}
\put(20,20){\line(3,-1){60}}
\put(-3,16){$\bullet$}
\end{picture}
 & 5 & 3 & /\\
 \hline 
 4.&  
\setlength{\unitlength}{0.333333mm}
\begin{picture}(60,25)(0,0)
\put(0,0){\line(0,1){20}}
\put(0,0){\line(1,0){60}}
\put(0,20){\line(1,0){40}}
\put(20,0){\line(0,1){20}}
\put(0,0){\line(1,1){20}}
\put(20,20){\line(2,-1){40}}
\put(20,20){\line(1,-1){20}}
\put(40,20){\line(1,-1){20}}
\put(17,16){$\bullet$}
\end{picture}
 & 5 & 3 & /\\
 \hline
  5.&
\setlength{\unitlength}{0.333333mm}
\begin{picture}(100,25)(0,0)
\put(0,0){\line(0,1){20}}
\put(0,0){\line(1,0){100}}
\put(0,20){\line(1,0){20}}
\put(0,0){\line(1,1){20}}
\put(20,0){\line(0,1){20}}
\put(20,20){\line(1,-1){20}}
\put(20,20){\line(2,-1){40}}
\put(20,20){\line(3,-1){60}}
\put(20,20){\line(4,-1){80}}
\put(17,16){$\bullet$}
\end{picture}\ \ \ \ \ 
\begin{picture}(110,25)(0,0)
\put(0,0){\line(0,1){20}}
\put(0,0){\line(1,0){100}}
\put(20,20){\line(4,-1){80}}
\put(0,20){\line(1,0){20}}
\put(0,20){\line(1,-1){20}}
\put(0,20){\line(2,-1){40}}
\put(0,20){\line(3,-1){60}}
\put(0,20){\line(4,-1){80}}
\put(0,20){\line(5,-1){100}}
\put(-3,16){$\bullet$}
\end{picture}
 & 6 & 6 &/  \\
 \hline
6.&
\setlength{\unitlength}{0.333333mm}
\begin{picture}(80,25)(0,0)
\put(0,0){\line(0,1){20}}
\put(0,0){\line(1,0){80}}
\put(0,20){\line(1,0){40}}
\put(20,0){\line(0,1){20}}
\put(0,0){\line(1,1){20}}
\put(20,20){\line(2,-1){40}}
\put(20,20){\line(1,-1){20}}
\put(40,20){\line(2,-1){40}}
\put(20,20){\line(3,-1){60}}
\put(17,16){$\bullet$}
\end{picture} & 6 & 6 &/  \\
\hline

 7.&
\begin{picture}(80,20)(0,-5) $S(1,\delta-1)$ \end{picture}
 & \begin{picture}(20,20)(0,-5)
 $ \delta\geq 7$  \end{picture} & \begin{picture}(20,20)(0,-5) $ {{\delta-2}\choose{2}}$ \end{picture} & \begin{picture}(5,20)(0,-5) / \end{picture} \\
 \hline 
 8.& 
 \begin{picture}(80,20)(0,-5) $S(2,\delta-2)$ \end{picture}
 & 
\begin{picture}(20,20)(0,-5)
 $ \delta\geq 7$  \end{picture} &  \begin{picture}(20,20)(0,-5) $ {{\delta-2}\choose{2}}$ \end{picture}  &  \begin{picture}(20,20)(0,-5) $ {{\delta-4}\choose{3}}$ \end{picture}  \\
\hline 
\end{tabular}
\caption{Rational singularities}\label{caso razionale sec 1}
\end{figure}

\begin{figure}[h!!]
\begin{tabular}{|l|l|c|c|c|}
\hline
\ &\ \ triangulation of $Q$\ \  & $\deg(Z)$ & $\nu_2(Z)$ & $\nu_3(Z)$ \\
\hline
 \hline
1.
&
\setlength{\unitlength}{0.6mm}
\begin{picture}(25,22.5)(0,0)
\put(0,10){\line(0,1){10}}
\put(0,10){\line(1,0){20}}
\put(10,0){\line(0,1){20}}
\put(0,10){\line(1,-1){10}}
\put(0,20){\line(1,-1){10}}
\put(0,20){\line(1,0){10}}
\put(10,0){\line(1,1){10}}
\put(10,20){\line(1,-1){10}}
\put(8.5,8.5){$\bullet$}
\end{picture} \ \ \ \ \ 
\begin{picture}(25,22.5)(0,0)
\put(0,10){\line(1,1){10}}
\put(0,10){\line(1,0){10}}
\put(10,0){\line(0,1){20}}
\put(0,0){\line(0,1){10}}
\put(20,0){\line(-1,1){10}}
\put(0,0){\line(1,0){20}}
\put(0,0){\line(1,1){10}}
\put(10,20){\line(1,-2){10}}
\put(8.5,8.5){$\bullet$}
\end{picture}
 & 5 & 1 & / \\ \hline 
 2.&
\setlength{\unitlength}{0.6mm}
\begin{picture}(25,22.5)(0,0)
\put(0,10){\line(0,1){10}}
\put(0,10){\line(1,0){20}}
\put(10,0){\line(0,1){20}}
\put(0,10){\line(1,-1){10}}
\put(0,20){\line(1,-1){20}}
\put(0,20){\line(1,0){10}}
\put(10,0){\line(1,0){10}}
\put(20,0){\line(0,1){10}}
\put(10,20){\line(1,-1){10}}
\put(8,8){$\bullet$}
\end{picture}\ \ \ \ \ 
\begin{picture}(25,22.5)(0,0)
\put(0,10){\line(1,1){10}}
\put(0,10){\line(1,0){20}}
\put(10,0){\line(0,1){20}}
\put(0,0){\line(0,1){10}}
\put(20,0){\line(-1,1){10}}
\put(0,0){\line(1,0){20}}
\put(20,0){\line(0,1){10}}
\put(0,0){\line(1,1){10}}
\put(10,20){\line(1,-1){10}}
\put(8.5,8.5){$\bullet$}
\end{picture}\ \ \ \ \ 
\begin{picture}(25,22.5)(0,0)
\put(0,0){\line(0,1){20}}
\put(0,10){\line(1,0){20}}
\put(10,0){\line(0,1){10}}
\put(0,20){\line(1,-1){20}}
\put(0,20){\line(2,-1){20}}
\put(0,0){\line(1,0){20}}
\put(20,0){\line(0,1){10}}
\put(0,0){\line(1,1){10}}
\put(8.5,8.5){$\bullet$}
\end{picture}\ \ \ \ \ 
\begin{picture}(35,22.5)(0,0)
\put(0,0){\line(0,1){20}}
\put(0,0){\line(1,0){30}}
\put(0,20){\line(3,-2){30}}
\put(0,0){\line(1,1){10}}
\put(0,10){\line(1,0){10}}
\put(0,20){\line(1,-1){10}}
\put(10,0){\line(0,1){10}}
\put(10,10){\line(1,-1){10}}
\put(10,10){\line(2,-1){20}}
\put(8.5,8.5){$\bullet$}
\end{picture}
 & 6 & 3 & / \\
 \hline
3.
&
\setlength{\unitlength}{0.6mm}
\begin{picture}(25,22.5)(0,0)
\put(0,0){\line(0,1){20}}
\put(0,10){\line(1,0){20}}
\put(10,0){\line(0,1){20}}
\put(0,20){\line(1,-1){20}}
\put(0,20){\line(1,0){10}}
\put(0,0){\line(1,0){20}}
\put(20,0){\line(0,1){10}}
\put(10,20){\line(1,-1){10}}
\put(0,0){\line(1,1){10}}
\put(8.5,8.5){$\bullet$}
\end{picture}\ \ \ \ \ 
\begin{picture}(35,22.5)(0,0)
\put(0,0){\line(0,1){10}}
\put(0,10){\line(1,0){20}}
\put(10,0){\line(0,1){20}}
\put(10,10){\line(1,-1){10}}
\put(0,10){\line(1,1){10}}
\put(0,0){\line(1,0){30}}
\put(30,0){\line(-1,1){10}}
\put(10,20){\line(1,-1){10}}
\put(0,0){\line(1,1){10}}
\put(10,10){\line(2,-1){20}}
\put(8.5,8.5){$\bullet$}
\end{picture}
 & 7 & 6 & / \\
 \hline 
4.
&
\setlength{\unitlength}{0.6mm}
\begin{picture}(25,22.5)(0,0)\put(0,0){\line(0,1){20}}
\put(0,0){\line(1,0){20}}
\put(20,0){\line(0,1){20}}
\put(0,20){\line(1,0){20}}
\put(10,0){\line(0,1){20}}
\put(0,10){\line(1,0){20}}
\put(0,20){\line(1,-1){20}}
\put(0,0){\line(1,1){20}}
\put(8.5,8.5){$\bullet$}
\end{picture}\ \ \ \ \ 
\begin{picture}(45,22.5)(0,0)
\put(0,0){\line(0,1){20}}
\put(0,0){\line(1,0){40}}
\put(0,10){\line(1,0){20}}
\put(0,20){\line(1,-1){20}}
\put(10,0){\line(0,1){10}}
\put(0,20){\line(2,-1){40}}
\put(0,0){\line(1,1){10}}
\put(10,10){\line(2,-1){20}}
\put(10,10){\line(3,-1){30}}
\put(8.5,8.5){$\bullet$}
\end{picture}
& 8 & 10 & / \\
 \hline
5.&
\setlength{\unitlength}{0.6mm}
\begin{picture}(35,22.5)(0,0)
\put(0,0){\line(0,1){20}}
\put(0,10){\line(1,0){20}}
\put(10,0){\line(0,1){20}}
\put(0,20){\line(1,-1){20}}
\put(0,20){\line(1,0){10}}
\put(0,0){\line(1,0){30}}
\put(30,0){\line(-1,1){10}}
\put(10,20){\line(1,-1){10}}
\put(0,0){\line(1,1){10}}
\put(10,10){\line(2,-1){20}}
\put(8.5,8.5){$\bullet$}
\end{picture}  
 & 8 & 10 & 1 \\
 \hline
6.&
\setlength{\unitlength}{0.6mm}
\begin{picture}(35,32.5)(0,0)
\put(0,0){\line(0,1){30}}
\put(0,0){\line(1,0){30}}
\put(0,30){\line(1,-1){30}}
\put(0,0){\line(1,1){10}}
\put(0,10){\line(1,0){20}}
\put(10,0){\line(0,1){20}}
\put(0,20){\line(1,-1){20}}
\put(10,10){\line(2,-1){20}}
\put(10,10){\line(-1,2){10}}
\put(8.5,8.5){$\bullet$}
\end{picture} 
 & 9 & 15 & 4 \\
 \hline
\end{tabular}
\caption{Elliptic singularities}\label{caso ellittico sec 1}
\end{figure}

\section*{Acknowledgements}
I would like to thank C. Ciliberto for introducing me to the problem of studying secant varieties and suggesting to me the idea of using toric degenerations as a tool to compute the secant degrees, while I was a PhD student of him. I also want to thank  B. Sturmfels for suggesting me to classify the delightful triangulations. 
Finally, I am deeply grateful to R. Piene for many helpful and stimulating discussions during the preparation of this paper.


\end{document}